\documentclass[11pt]{amsart}

\usepackage{amssymb}

\usepackage[T1]{fontenc}
\usepackage{lmodern}

\usepackage{mathrsfs}

\usepackage[all]{xy}
\usepackage[headings]{fullpage}
\usepackage{paralist}
\usepackage{xspace}

\usepackage{color}
\usepackage[normalem]{ulem}

\usepackage[colorlinks=true,linkcolor=blue, citecolor=blue, urlcolor=blue, 
pagebackref=true]{hyperref}

\makeatletter

\newif\ifreadkumminibib
\readkumminibibfalse

\def\theenumi{\@alph\c@enumi}

\makeatother

\theoremstyle{plain}
\newtheorem{theorem}[equation]{Theorem}
\newtheorem{lemma}[equation]{Lemma}
\newtheorem{corollary}[equation]{Corollary}
\newtheorem{proposition}[equation]{Proposition}

\theoremstyle{definition}

\newtheorem{conjecture}[equation]{Conjecture}

\newtheorem{remark}[equation]{Remark}
\newenvironment{remarkbox}[1][]{%
    \begin{remark}[#1] \pushQED{\qed}}{\popQED \end{remark}}

\newtheorem{example}[equation]{Example}
\newenvironment{examplebox}[1][]{%
    \begin{example}[#1] \pushQED{\qed}}{\popQED \end{example}}

\newtheorem{definition}[equation]{Definition}

\newtheorem{notation}[equation]{Notation}
\newenvironment{notationbox}[1][]{%
    \begin{notation}[#1]\pushQED{\qed}}{\popQED \end{notation}}

\newtheorem{discussion}[equation]{Discussion}
\newenvironment{discussionbox}[1][]{%
    \begin{discussion}[#1]\pushQED{\qed}}{\popQED \end{discussion}}

\newtheorem{observation}[equation]{Observation}

\newtheorem{construction}[equation]{Construction}

\newtheorem{setup}[equation]{Setup}

\newcommand{\bbA}{\mathbb A}

\newcommand{\bfE}{\mathbf E}

\newcommand{\calI}{\mathscr I}

\newcommand{\frakm}{{\mathfrak m}}

 \let\strSh\calO

\newcommand{\calQ}{\mathcal Q}

\newcommand{\calR}{\mathcal R}

\newcommand{\calU}{\mathcal U}
\newcommand{\calV}{\mathcal V}
\newcommand{\calW}{\mathcal W}

\DeclareMathOperator{\Schur}{S}

\let\mhyphen=-

\newcommand{\ints}{\mathbb{Z}}

\newcommand{\complex}{\mathbb{C}}

\def\to{\longrightarrow}
\def\xyhookrightarrownotip{\ar@{^(-}}
\def\xyhookrightarrow{\ar@{^(->}}

\DeclareMathOperator{\rank}{rk}

\DeclareMathOperator{\projective}{\mathbb{P}}
\DeclareMathOperator{\reg}{reg}

\DeclareMathOperator{\Tor}{Tor}

\DeclareMathOperator{\Grass}{Gr}
\DeclareMathOperator{\depth}{depth}
\DeclareMathOperator{\homology}{H}

\newcommand{\define}[1]{\emph{#1}}
\makeatletter
\def\RDerChar{\mathbf{R}}
\def\RDer{\@ifnextchar[{\R@Der}{\ensuremath{\RDerChar}}}
\def\R@Der[#1]{\ensuremath{\RDerChar^{#1}}}
\makeatother
\newcommand{\excise}[1]{}

\newcommand{\GL}{\mathrm{GL}}

\newcommand{\linearity}{the linearity property\xspace}

\title{Free resolutions of some Schubert singularities.}

\author[M.~Kummini]{Manoj Kummini}
\address{Chennai Mathematical Institute, Siruseri, Tamilnadu 603103. India}
\email{mkummini@cmi.ac.in}

\author{V. Lakshmibai}
\address{Northeastern University, Boston, Massachusetts. USA.}
\email{lakshmibai@neu.edu}

\author[P.~Sastry]{Pramathanath Sastry}
\address{Chennai Mathematical Institute, Siruseri, Tamilnadu 603103. India}
\email{pramath@cmi.ac.in}

\author{C.~S.~Seshadri}
\address{Chennai Mathematical Institute, Siruseri, Tamilnadu 603103. India}
\email{css@cmi.ac.in}

\thanks{The first author was supported by a CMI Faculty Development Grant.
The second author was supported by NSA grant H98230-11-1-0197, NSF grant 
0652386.}

\begin{document}

\begin{abstract}
In this paper we construct free resolutions of certain class of closed subvarieties
of affine spaces (the so-called ``opposite big cells'' of Grassmannians).
Our class covers the determinantal varieties, whose resolutions were first
constructed by A.~Lascoux (Adv. Math., 1978). Our approach uses the
geometry of Schubert varieties.  An interesting aspect of our work is its
connection to the computation of the cohomology of homogeneous bundles
(that are not necessarily completely reducible) on partial flag varieties.
\end{abstract}

\maketitle

\section{Introduction}
\label{sec:intro}
\numberwithin{equation}{section}

A classical problem in commutative algebra and algebraic geometry is to
describe the syzygies of defining ideals of interesting varieties.
Let $k \leq n \leq m$ be positive integers. The space $D_k$ of $m \times n$
matrices (over a field $\Bbbk$) of rank at most $k$
is a closed subvariety of the $mn$-dimensional affine space of all $m
\times n$ matrices. When $\Bbbk = \complex$,  a minimal free resolution of the coordinate ring
$\Bbbk[\strSh_{D_k}]$ as a
module over the coordinate ring of the the $mn$-dimensional affine space
(i.e. the $mn$-dimensional polynomial ring) was constructed by
A.~Lascoux~\cite {LascSyzVarDet78}; see also 
\cite[Chapter~6]{WeymSyzygies03}.

In this paper, we construct free resolutions for a larger class of
singularities, \textit{viz.},
Schubert singularities, i.e., the intersection of a singular Schubert
variety and the ``opposite big cell'' inside a Grassmannian.
The advantage of our method is that it is algebraic group-theoretic, and is
likely to work for Schubert singularities in more general flag varieties.
In this process, we have come up with a method to compute the cohomology of
certain homogeneous vector-bundles (which are not completely reducible) on
flag varieties.  We will work over $\Bbbk = \complex$.

Let $N = m+n$. Let $\GL_N = \GL_N(\complex)$ be the group of $N\times N$
invertible matrices. Let $B_N$ be the Borel subgroup of
all upper-triangular matrices  and $B_N^\mhyphen$ the \define{opposite} Borel
subgroup of all lower-triangular matrices in $\GL_N$.
Let $P$ be the maximal parabolic subgroup
corresponding to omitting the simple root $\alpha_{n}$, i.e, the
subgroup of $\GL_N$ comprising the matrices in which the $(i,j)$-th entry
(i.e., in row $i$ and column $j$) is zero, if $n+1 \leq i \leq N$ and $1
\leq j \leq n$; in other words,
\[
P = \left\{
\begin{bmatrix} A_{n\times n} & C_{n \times m} \\
0_{m\times n} & E_{m \times m} \end{bmatrix} \in \GL_N\right\}.
\]
We have a canonical identification of the Grassmannian of $n$-dimensional
subspaces of $\Bbbk^N$ with $\GL_N/P$.
Let $W$ and $W_P$ be the Weyl groups of $\GL_N$ and of $P$,
respectively; note that $W = S_N$ (the symmetric group) and $W_P
= S_n \times S_m$.
For $w \in W/W_P$, let $X_P(w) \subseteq \GL_N/P$ be the Schubert variety 
corresponding to $w$
(i.e., the closure of the $B_N$-orbit of the coset  $wP$ ($\in\GL_N/P$), equipped with the canonical reduced scheme structure).
 The $B_N^\mhyphen$-orbit of the
coset $(\mathrm{id}\cdot P)$
in $\GL_N/P$ is denoted by $O^\mhyphen_{\GL_N/P}$, and is usually called the \emph{opposite big
cell} in $\GL_N/P$; it can be identified with the $mn$-dimensional affine
space.
(See~\ref{sec:precisGLn}.)

Write $W^P$ for the set of minimal representatives (under the Bruhat order)
in $W$ for the elements of $W/W_P$. 
For $1 \leq r \leq n-1$, we consider certain subsets
$\calW_r$  of $W^P$ (Definition~\ref{definition:ClassWk}); there is $w \in 
\calW_{n-k}$ such that  $D_k = X_P(w) \cap O^\mhyphen_{\GL_N/P}$.
Note that for any $w \in W^P$, $X_P(w) \cap O^\mhyphen_{\GL_N/P}$ is a closed
subvariety of $O^\mhyphen_{\GL_N/P}$.
Our main result is a description of the minimal free
resolution of the coordinate ring of $X_P(w) \cap O^\mhyphen_{\GL_N/P}$ as a module
over the coordinate ring of $O^\mhyphen_{\GL_N/P}$ for every $w \in \calW_r$. This
latter ring is a polynomial ring. 
We now outline our approach.

First we recall the Kempf-Lascoux-Weyman ``geometric technique'' of
constructing minimal free resolutions.
Suppose that we have a commutative diagram of varieties \begin{equation}
\label{equation:genericKLW}
\vcenter{\vbox{%
\xymatrix{%
Z \xyhookrightarrow[r] \ar[d]^{q'} & \bbA  \times V \ar[d]^q \ar[r] & V \\
Y \xyhookrightarrow[r] & \bbA }}}
\end{equation}
where $\bbA$ is an affine space, $Y$ a closed subvariety of $\bbA$ and $V$ 
a
projective variety.
The map $q$ is first projection, $q'$ is proper and
birational,  and the inclusion $Z \hookrightarrow \bbA  \times V$ is a 
sub-bundle (over $V$) of the trivial bundle $\bbA  \times V$.
Let $\xi$ be the dual of the quotient bundle on $V$ corresponding to $Z$.
Then the derived direct image $\RDer q'_* \strSh_{Z}$ is quasi-isomorphic 
to a minimal complex
$F_\bullet$ with
\[
F_i = \oplus_{j \geq 0} \homology^j(V, \bigwedge^{i+j} \xi)
\otimes_\complex R(-i-j).
\]
Here $R$ is the coordinate ring of $\bbA$; it is a polynomial ring
and
$R(k)$ refers to twisting with respect to its natural grading.
If $q'$ is such that the natural map $\strSh_Y \to \RDer q'_* \strSh_{Z}$
is a quasi-isomorphism,
(for example, if $q'$ is a desingularization of $Y$ and $Y$ has rational
singularities), then $F_\bullet$  is a minimal free resolution of
$\complex[Y]$ over the polynomial ring $R$ .

The difficulty in applying this technique in any given situation is
two-fold: one must find a suitable morphism $q' : Z \to Y$ such
that the map $\strSh_Y \to \RDer q'_* \strSh_{Z}$
is a quasi-isomorphism and such that $Z$ is
a vector-bundle over a projective variety $V$; and, one must be
able to compute the necessary cohomology groups.
We overcome this for opposite cells in a certain class (which includes
the determinantal varieties) of Schubert varieties in a Grassmannian, in
two steps.

As the first step, we need to establish the existence of a diagram as above. This
is done using the geometry of Schubert varieties. 
We take $\bbA = O^\mhyphen_{\GL_N/P}$ and $Y = Y_P(w) := X_P(w) \cap O^\mhyphen_{\GL_N/P}$.
Let $\tilde P$ be a parabolic subgroup with $B_N \subseteq \tilde P 
\subsetneq P$.
The inverse image of $O^\mhyphen_{\GL_N/P}$ under
the natural map $\GL_N/\tilde P \to \GL_N/P$ is $O^\mhyphen_{\GL_N/P} \times P/\tilde 
P$. 
Let $\tilde w$ be the representative of
the coset $w\tilde P$ in $W^{\tilde P}$. Then $X_{\tilde P}(\tilde w)
\subseteq \GL_N/\tilde P$ (the Schubert subvariety of $\GL_N/\tilde P$ 
associated to $\tilde w$)
maps properly and birationally onto $X_{P}(w)$.
We may choose $\tilde P$ to ensure that $X_{\tilde P}(\tilde w)$ is 
smooth. 
%
%
Let
$Z_{\tilde{P}}(\tilde w)$ be the preimage of $Y_P(w)$ in
$X_{\tilde{P}}(\tilde w)$. We take $Z = Z_{\tilde{P}}(\tilde w)$. Then $V$, 
which is the image of $Z$ under the
second projection, is a smooth Schubert subvariety of $P/\tilde P$.
The vector-bundle $\xi$ on $V$ that we obtain is the restriction of a
homogeneous bundle on $P/\tilde P$.
Thus we get:
\begin{equation}
\label{equation:KLWdiagram}
\vcenter{\vbox{%
\xymatrix{%
Z_{\tilde P}(\tilde w) \ar[d]^{q'} \xyhookrightarrow[r] & O^\mhyphen_{\GL_N/P} \times 
V  \ar[d]^{q} \ar[r] & V\\
Y_{P}(w)  \xyhookrightarrow[r] & O^\mhyphen
}}}.
\end{equation}
See Theorem~\ref{theorem:birational} and
Corollary~\ref{corollary:ourRealization}.
In this diagram, $q'$ is a desingularization of $Y_P(w)$. Since it is known
that Schubert varieties have rational singularities, we have that the map 
$\strSh_Y \to \RDer q'_* \strSh_{Z}$ is a quasi-isomorphism, so
$F_\bullet$ is a minimal resolution.

As the second step, we need to determine the cohomology of the homogeneous
bundles $\wedge^t \xi$ over $V$. There are two ensuing issues: computing 
cohomology of  homogeneous vector-bundles over Schubert
subvarieties of flag varieties is difficult and, furthermore, these bundles
are not usually completely reducible, so one cannot apply the
Borel-Weil-Bott theorem directly.  We address the former issue by
restricting our class; if $w \in \calW_r$ (for some $r$) then $V$ will
equal $P/\tilde P$.  
Regarding the latter issue, we inductively replace $\tilde P$ by larger
parabolic subgroups (still inside $P$), such that at each stage, the
computation reduces to that of the cohomology of completely reducible
bundles on Grassmannians; using various spectral sequences, we are able
to determine the cohomology groups that determine the minimal free
resolution.
See Proposition~\ref{proposition:higherdirectimageSpecific} for the key
inductive step.
In contrast, in Lascoux's construction of the resolution of determinantal
ideals, one comes across only completely reducible bundles; therefore, one
may use the Borel-Weil-Bott theorem to compute the cohomology of the 
bundles $\wedge^t\xi$.

Computing cohomology of homogeneous bundles, in
general, is difficult, and is of independent interest; we hope that our
approach would be useful in this regard.
The best results, as far as we know, are due to G.~Ottaviani and
E.~Rubei~\cite {OttavianiRubeiQuivers2006}, which deal with general
homogeneous bundles on Hermitian symmetric spaces. The only Hermitian
symmetric spaces in Type A are the Grassmannians, so their results do not
apply to our situation.

Since the opposite big cell $O^\mhyphen_{\GL_N/P}$ intersects every $B_N$-orbit of
$\GL_N/P$, $Y_P(w)$ captures all the singularities of $X_P(w)$ for every $w 
\in W$. In this paper, we describe a construction of
a minimal free resolution of $\complex[Y_P(w)]$ over
$\complex[O^\mhyphen_{\GL_N/P}]$. We hope that our methods could shed some light on
the problem of construction of a locally free resolution of 
$\strSh_{X_P(w)}$
as an $\strSh_{\GL_N/P}$-module.

The paper is organized as follows. 
Section~\ref{sec:prelims} contains 
notations and conventions (Section~\ref{sec:notation}) and the necessary
background material on Schubert varieties
(Section~\ref{sec:precisGLn}) and
homogeneous bundles (Section~\ref{sec:homogvb}).
In Section~\ref{sec:schubertdesing}, we discuss properties of Schubert
desingularization, including the construction of
Diagram~\eqref{equation:KLWdiagram}. 
Section~\ref{sec:freeresolutions} is devoted to a review
of the Kempf-Lascoux-Weyman technique and its application to our problem.
Section~\ref {sec:steptwo} explains how the cohomology of the homogeneous
bundles on certain partial flag varieties can be computed;
Section~\ref {sec:examples} gives some examples.
Finally, in Section~\ref{sec:furtherrmks}, we describe Lascoux's resolution
in terms of our approach and describe the multiplicity and
Castelnuovo-Mumford regularity of $\complex[Y_P(w)]$.

\section*{Acknowledgements}
Most of this work was done during a visit of the first author to
Northeastern University and the visits of the second author to Chennai
Mathematical Institute; the two authors would like to thank the respective
institutions for the hospitality extended to them during their visits.
The authors thank V.~Balaji, Reuven Hodges and A.~J.~Parameswaran for 
helpful comments.
The computer algebra
systems \texttt{Macaulay2}~\cite{M2} and \texttt{LiE}~\cite{LiE} provided
valuable assistance in studying examples.

\section{Preliminaries}
\label{sec:prelims}
\numberwithin{equation}{subsection}

In this section, we collect various results about Schubert varieties,
homogeneous bundles and the Kempf-Lascoux-Weyman geometric technique.

\subsection{Notation and conventions}
\label{sec:notation}

We collect the symbols used and the conventions adopted in the rest of the
paper here.  For details on algebraic groups and Schubert varieties, the
reader may refer to ~\cite {BorelLinAlgGps91, JantzenReprAlgGps2003,
BilleyLakshmibaiSingularLoci2000, SeshadriSMT07}.

Let $m \geq n$ be positive integers and $N = m+n$.
We denote by $\GL_N$ (respectively,  
$B_N$, $B_N^\mhyphen$) the group of all (respectively, upper-triangular,
lower-triangular) invertible 
$N\times N$ matrices over $\complex$. 
The Weyl group $W$ of $\GL_N$ is isomorphic to the group $S_N$ of
permutations of $N$ symbols and is generated by the \define{simple
reflections} $s_i, 1 \leq i \leq N-1$, which correspond to the
transpositions $(i,i+1)$. 
For $w \in W$, its \define{length} is the 
smallest integer $l$ such that $w = s_{i_1}\cdots s_{i_l}$ as a
product of simple reflections.
For every $1 \leq i \leq N-1$, there is a
\define{minimal parabolic subgroup} $P_i$ containing $s_i$ (thought of as
an element of $\GL_N$)
and a \define{maximal parabolic subgroup} $P_{\widehat{i}}$ \emph{not
containing} $s_i$. Any parabolic subgroup can be written as 
$P_{\widehat{A}} := \bigcap_{i \in A} P_{\widehat{i}}$ for some $A \subset 
\{1, \ldots, N-1\}$. On the other hand, for $A \subseteq \{1, \ldots, 
N-1\}$ write $P_A$ for the subgroup of $\GL_N$ generated by $P_i, i \in A$. 
Then $P_A$
is a parabolic subgroup and $P_{\{1, \ldots, N-1\} \setminus A}=
P_{\widehat{A}}$.

The following is fixed for the rest of this paper:
\begin{enumerate}
\item $P$ is the maximal parabolic subgroup $P_{\widehat{n}}$ of
$\GL_N$;

\item for $1 \leq s \leq n-1$, $\tilde{P}_s$ is the parabolic
subgroup $P_{\{1, \ldots, s-1, n+1, \ldots, N-1\}} = \cap_{i=s}^n
P_{\widehat{i}}$ of $\GL_N$;

\item for $1 \leq s \leq n-1$, $Q_s$ is the parabolic
subgroup $P_{\{1, \ldots, s-1\}} = \cap_{i=s}^{n-1} P_{\widehat{i}}$ of 
$\GL_n$.
\end{enumerate}

We write the elements of $W$ in \define{one-line} notation: $(a_1,
\ldots, a_N)$ is the permutation $i \mapsto a_i$.  For any $A \subseteq 
\{1, \ldots, N-1\}$, define $W_{P_A}$ to be the
subgroup of $W$ generated by $\{s_i : i \in
A\}$. By $W^{P_A}$ we mean the subset of $W$ consisting of the minimal
representatives (under the Bruhat order) in $W$ of the elements of
$W/W_{P_A}$.
For $1 \leq i \leq N$, we represent the elements of
$W^{P_{\widehat{i}}}$ by sequences $(a_1, \ldots, a_i)$ with $1 \leq a_1 <
\cdots < a_i \leq N$ since under the action of the group
$W_{P_{\widehat{i}}}$, every element of $W$ can be represented minimally by
such a sequence.

For $w = (a_1, a_2, \ldots, a_n) \in W^P$, let $r(w)$ be the integer $r$ 
such that $a_r \leq n < a_{r+1}$.

We identify $\GL_N = \GL(V)$ for some $N$-dimensional vector-space $V$.  
Let $A := \{i_1 < i_2 < \cdots < i_r\} \subseteq \{1, \ldots, N-1\}$.  Then 
$\GL_N/P_{\widehat{A}}$ is the set of all flags $0 = V_0 \subsetneq V_{1}
\subsetneq V_2 \subsetneq  \cdots \subsetneq V_r \subsetneq V$ of subspaces
$V_j$ of dimension $i_j$ inside $V$. We call $\GL_N/P_{\widehat{A}}$ a
\define{flag variety}. If $A = \{1, \ldots, N-1\}$ (i.e.  $P_{\widehat{A}} 
= B_N$), then we call the the flag variety a \define{full flag
variety}; otherwise, a \define{partial flag variety}.
The \define{Grassmannian} $\Grass_{i,N}$ of $i$-dimensional subspaces of 
$V$
is $\GL_N/P_{\widehat{i}}$. 

Let ${\tilde{P}}$ be any parabolic subgroup containing $B_N$ 
and $\tau \in W$.  The \define{Schubert
variety} $X_{{\tilde{P}}}(\tau)$ is the closure inside $\GL_N/{\tilde{P}}$ of $B_N\cdot e_w$
where $e_w$ is the coset $\tau{\tilde{P}}$, endowed with the canonical reduced
scheme structure.  Hereafter, when we write $X_{\tilde{P}}(\tau)$, we mean that $\tau$ is 
the representative in $W^{\tilde{P}}$ of its coset.
The \define{opposite big cell} $O^\mhyphen_{\GL_N/{\tilde{P}}}$ in 
$\GL_N/{\tilde{P}}$ is  the $B_N^\mhyphen$-orbit of the coset 
$(\mathrm{id}\cdot {\tilde{P}})$ in $\GL_N/{\tilde{P}}$. 
Let $Y_{\tilde{P}}(\tau) := X_{\tilde{P}}(\tau) \cap 
O^\mhyphen_{\GL_N/{\tilde{P}}}$; we refer to $Y_{\tilde{P}}(\tau)$ as the 
\define{opposite cell} of $X_{\tilde{P}}(\tau)$.

We will write $R^+$, $R^\mhyphen$, $R^+_{\tilde{P}}$,
$R^\mhyphen_{\tilde{P}}$, to denote respectively, positive and negative roots for $\GL_N$
and for $\tilde{P}$.
We denote by $\epsilon_i$ the character that sends the invertible diagonal
matrix with $t_1, \ldots, t_n$ on the diagonal to $t_i$.

\subsection{Pr\'ecis on $\GL_n$ and Schubert varieties}
\label{sec:precisGLn}

Let $\tilde{P}$ be a parabolic subgroup of $\GL_N$ 
with $B_N \subseteq \tilde{P} \subseteq P$.
We will use the following proposition extensively in the sequel. 

\begin{proposition}
\label{proposition:structureOMinus}
Write $U^\mhyphen_{\tilde{P}}$ for the negative unipotent radical of
$\tilde{P}$.
\begin{asparaenum}
\item \label{proposition:structureOMinusUminus}
$O^\mhyphen_{\GL_N/\tilde{P}}$ can be naturally identified with 
$U^\mhyphen_{\tilde{P}}\tilde{P}/\tilde{P}$. 
\item \label{proposition:structureOMinusinvertibleA}
For
\[
z = \begin{bmatrix}
A_{n \times n} & C_{n \times m} \\
D_{m \times n} & E_{m \times m}
\end{bmatrix} \in \GL_N,
\]
$zP \in O^\mhyphen_{\GL_N/P}$ if and only if $A$ is invertible.
\item \label{proposition:structureOMinusInverseImage}
For $1 \leq s \leq n-1$, the inverse image of $O^\mhyphen_{\GL_N/P}$ under the
natural map $\GL_N/{\tilde{P}_s} \to \GL_N/P$ is isomorphic to $O^\mhyphen_{\GL_N/P} 
\times P/{\tilde{P}_s}$ as schemes. Every element of $O^\mhyphen_{\GL_N/P} \times 
P/{\tilde{P}_s}$ is of the form
\[
\begin{bmatrix}
A_{n \times n} & 0_{n \times m} \\
D_{m \times n} & I_{m}
\end{bmatrix} \mod \tilde{P}_s \in \GL_N/\tilde{P}_s.
\]
Moreover, two matrices
\[
\begin{bmatrix}
A_{n \times n} & 0_{n \times m} \\
D_{m \times n} & I_{m}
\end{bmatrix} \;\text{and}\;
\begin{bmatrix}
A'_{n \times n} & 0_{n \times m} \\
D'_{m \times n} & I_{m}
\end{bmatrix}
\]
in $\GL_N$ represent the same element modulo $\tilde{P}_s$ if and only if
there exists a matrix $q \in Q_s$ such that $A' = Aq$ and $D' = Dq$.
\item \label{proposition:SLnModQ}
For $1 \leq s \leq n-1$, $P/{\tilde{P}_s}$ is isomorphic to $\GL_n/Q_s$.
In particular, the projection map $O^\mhyphen_{\GL_N/P} \times P/{\tilde{P}} \to 
P/{\tilde{P}_s}$
is given by
\[
\begin{bmatrix}
A_{n \times n} & 0_{n \times m} \\
D_{m \times n} & I_{m}
\end{bmatrix} \mod \tilde{P}_s \mapsto A \mod \tilde{Q} \in
\GL_n/Q \simeq P/\tilde{P}_s.
\]
\end{asparaenum}
\end{proposition}

\begin{proof}
\underline{\eqref{proposition:structureOMinusUminus}}:
Note that $U^\mhyphen_{\tilde{P}}$ is the subgroup of $\GL_N$ generated by
the (one-dimensional) root subgroups $U_{\alpha}, \alpha \in R^\mhyphen
\setminus R_{\tilde{P}}^\mhyphen$ and that 
$U^\mhyphen_{\tilde{P}}\tilde{P}/\tilde{P} = 
B_N^\mhyphen\tilde{P}/\tilde{P}$. Hence 
under the canonical projection $\GL_N \to \GL_N/P,
g\mapsto gP$, $U^\mhyphen_P$ is mapped onto
$O^\mhyphen_{\GL_N/\tilde{P}}$. It is easy to check that this is an
isomorphism.

\underline{\eqref{proposition:structureOMinusinvertibleA}}:
Suppose that $zP \in O^\mhyphen_{\GL_N/P}$.
By~\eqref{proposition:structureOMinusUminus}, we see that there exist 
matrices $A'_{n \times n}$, $C'_{n \times m}$, $D'_{m \times n}$ and $E'_{m 
\times m}$ such that \[
z_1 := \begin{bmatrix}
I_{n} & 0_{n \times m} \\
D'_{m \times n} & I_{m}
\end{bmatrix} \in U^\mhyphen_P, \;
z_2 := \begin{bmatrix}
A'_{n \times n} & C'_{n \times m} \\
0_{m \times n} & E'_{m \times m}
\end{bmatrix} \in P \;\text{and}\;
z = \begin{bmatrix}
A_{n \times n} & C_{n \times m} \\
D_{m \times n} & E_{m \times m}
\end{bmatrix} = z_1z_2.
\]
Hence $A = A'$ is invertible. Conversely, if $A$ is invertible, then we may
write $z = z_1z_2$ where
\[
z_1 := \begin{bmatrix}
I_n & 0 \\
DA^{-1} & I_{m}
\end{bmatrix} \in U^\mhyphen_P \;\text{and}\;
z_2 := \begin{bmatrix}
A & C \\
0 & E-DA^{-1}C
\end{bmatrix}.
\]
Since $z \in \GL_N$, $z_2 \in P$.

\underline{\eqref{proposition:structureOMinusInverseImage}}:
Let $z \in U^\mhyphen_PP \subseteq \GL_N$. Then we can write $z = z_1z_2$ 
uniquely with $z_1 \in U^\mhyphen_P$ and $z_2 \in P$.
For, suppose that
\[
\begin{bmatrix}
I_{n} & 0_{n \times m} \\
D_{m \times n} & I_{m}
\end{bmatrix}
\begin{bmatrix}
A_{n \times n} & C_{n \times m} \\
0_{m \times n} & E_{m \times m}
\end{bmatrix}
=
\begin{bmatrix}
I_{n} & 0_{n \times m} \\
D'_{m \times n} & I_{m}
\end{bmatrix}
\begin{bmatrix}
A'_{n \times n} & C'_{n \times m} \\
0_{m \times n} & E'_{m \times m}
\end{bmatrix}
\]
then $A=A'$, $C=C'$, $DA = D'A'$ and $DC+E = D'C'+E'$,
which yields that $D' = D$ (since $A=A'$ is invertible,
by~\eqref{proposition:structureOMinusinvertibleA}) and $E = E'$.
Hence $U^\mhyphen_P \times_{\complex} P = U^\mhyphen_PP$.
Therefore, for any parabolic subgroup $P' \subseteq P$, $U^\mhyphen_P 
\times_{\complex} P/P' = U^\mhyphen_PP/P'$.
The asserted isomorphism now follows by taking $P' = \tilde{P}_s$.

For the next statement, let \[
\begin{bmatrix}
A_{n \times n} & C_{n \times m} \\
D_{m \times n} & E_{m \times m}
\end{bmatrix} \in \GL_N
\]
with $A$ invertible (which we may assume
by~\eqref{proposition:structureOMinusinvertibleA}). Then we have a
decomposition (in $\GL_N$)
\[
\begin{bmatrix} A & C \\ D & E \end{bmatrix} = \begin{bmatrix} A & 
0_{n\times m} \\ D & I_m \end{bmatrix} \begin{bmatrix} I_n & A^{-1}C \\ 
0_{m\times n} & E-DA^{-1}C \end{bmatrix}.
\]
Hence 
\[
\begin{bmatrix} A & C \\ D & E \end{bmatrix} \equiv
\begin{bmatrix} A & 0_{n\times m} \\ D & I_m \end{bmatrix}  \mod 
\tilde{P}_s.
\]

Finally,
\[
\begin{bmatrix}
A_{n \times n} & 0_{n \times m} \\
D_{m \times n} & I_{m}
\end{bmatrix} \equiv
\begin{bmatrix}
A'_{n \times n} & 0_{n \times m} \\
D'_{m \times n} & I_{m}
\end{bmatrix} \mod \tilde{P}_s
\]
if and only if there exist matrices $q \in Q_s$, $q'_{n \times m}$  and
$\tilde q_{n \times n} \in \GL_m$ such that \[
\begin{bmatrix} A' & 0 \\ D' & I \end{bmatrix} =
\begin{bmatrix} A & 0 \\ D & I \end{bmatrix} \begin{bmatrix} q & q' \\ 0_{m 
\times n} & \tilde q
\end{bmatrix}, \]
which holds if and only if
$q'=0$, $\tilde q = I_m$, $A' = Aq$ and $D' = Dq$ (since
$A$ and $A'$ are invertible).

\underline{\eqref{proposition:SLnModQ}}: There is a surjective morphism of
$\complex$-group schemes $P
\to \GL_n$,
\[
\begin{bmatrix}
A_{n \times n} & C_{n \times m} \\
0_{m \times n} & E_{m \times m}
\end{bmatrix} \longmapsto A.
\]
This induces the required isomorphism. Notice that the element
\[
\begin{bmatrix}
A_{n \times n} & C_{n \times m} \\
D_{m \times n} & E_{m \times m}
\end{bmatrix} \mod \tilde{P}_s \in O^\mhyphen_{\GL_N/P} \times P/\tilde{P}_s
\]
decomposes (uniquely) as
\[
\begin{bmatrix}
I_{n} & 0 \\
DA^{-1} & I_{m}
\end{bmatrix}
\left(
\begin{bmatrix} A & C \\ 0 & E
\end{bmatrix} \mod \tilde{P}_s \right) \]
Hence it is mapped to $A \mod Q_s \in \GL_n/Q_s$. Now
use~\eqref{proposition:structureOMinusInverseImage}.
\end{proof}

\begin{discussionbox}
\label{discussionbox:Ominus}
Let $\tilde{P} = P_{\widehat{\{i_1, \ldots, i_t\}}}$ with $1 \leq i_1 <
\cdots < i_t \leq N-1$. Then using
Proposition~\ref{proposition:structureOMinus}%
\eqref{proposition:structureOMinusUminus} and its proof, 
$O^\mhyphen_{\GL_N/\tilde{P}}$ can be
identified with the affine space of lower-triangular matrices with possible
non-zero entries $x_{ij}$ at row $i$ and column $j$ where $(i,j)$ is such
that there exists 
$l \in \{i_1, \ldots, i_t\}$ such that $j \leq l < i \leq N$. 
To see this, note (from the proof of
Proposition~\ref{proposition:structureOMinus}%
\eqref{proposition:structureOMinusUminus}) that we are interested in those
$(i,j)$ such that the root $\epsilon_i - \epsilon_j$ belongs to 
$R^\mhyphen \setminus R_{\tilde{P}}^\mhyphen$. Since 
$R_{\tilde{P}}^\mhyphen = \bigcap_{k=1}^t R_{P_{\widehat{i_k}}}^\mhyphen$,
we see that we are looking for 
$(i,j)$ such that $\epsilon_i - \epsilon_j \in
R^\mhyphen \setminus R_{P_{\widehat{l}}}^\mhyphen$, for some $l \in 
\{i_1, \ldots, i_t\}$. For the maximal parabolic group $P_{\widehat{l}}\,$,
we have, 
$R^\mhyphen \setminus R_{P_{\widehat{l}}}^\mhyphen = 
\{\epsilon_i-\epsilon_j \mid 1\le j\le l<i\le N\}$.
Hence $\dim O^\mhyphen_{\GL_N/\tilde{P}} =
|R^\mhyphen\setminus R^\mhyphen_{\tilde{P}}|$.

Let $\alpha = \epsilon_i-\epsilon_j
\in R^\mhyphen\setminus R^\mhyphen_{\tilde{P}}$ and 
$l \in \{i_1, \ldots, i_t\}$.
Then the Pl\"ucker co-ordinate $p_{s_{\alpha}}^{(l)}$ on the Grassmannian
$\GL_N/P_{\widehat{l}}$ lifts to a regular function on 
${\GL_N/\tilde{P}}$, which we denote by the same symbol. Its restriction to 
${O^\mhyphen_{G/{\tilde{P}}}}$ is the 
the $l \times l$-minor with column indices $\{1, 2, \ldots, l\}$ and 
row indices $\{1,\ldots, j-1, j+1, \ldots, l, i\}$.
In particular, 
\begin{equation}
\label{equation:minorcoordinate}
x_{ij} = p_{s_{\alpha}}^{(j)}|_{O^\mhyphen_{G/{\tilde{P}}}}
\;\text{for every}\; \alpha = \epsilon_i-\epsilon_j \in 
R^\mhyphen\setminus R^\mhyphen_{\tilde{P}}.
\end{equation}
In general $p_{s_{\alpha}}^{(l)}$ need not be a linear form, or even
homogeneous; see the example discussed after
Definition~\ref{definition:linear}.
\end{discussionbox}

\begin{examplebox}
\label{examplebox:Ominus}
Figure~\ref{figure:Ominus} shows the shape of 
$O^\mhyphen_{\GL_N/\tilde{P}_s}$ for some $1 \leq s \leq n-1$. The
rectangular region labelled with a circled D is 
$O^\mhyphen_{\GL_N/P}$. The trapezoidal region 
labelled with a circled A is $O^\mhyphen_{P/\tilde{P}_s}$.
In this case, the $x_{ij}$ appearing in~\eqref{equation:minorcoordinate} 
are exactly those in the regions labelled A and B.
\begin{figure}
\setlength{\unitlength}{.5pt}
\begin{picture}(300,300)
\put(0,0){\line(0,1){300}}
\put(0,0){\line(1,0){10}}
\put(0,300){\line(1,0){10}}
\put(300,0){\line(0,1){300}}
\put(300,0){\line(-1,0){10}}
\put(300,300){\line(-1,0){10}}
\put(120,0){\line(0,1){300}}
\put(0,180){\line(1,0){300}}
\put(0,240){\line(1,0){60}}
\put(60,240){\line(1,-1){60}}
\put(60,240){\line(0,1){60}}

\put(25,270){{${I_s}$}}
\put(80,240){\large{$\mathbf{0}$}}
\put(205,240){\large{$\mathbf{0}$}}
\put(205,85){\large{${I_m}$}}

\put(45,210){\circle{32}}
\put(38,205){\large{$A$}}
\put(55,90){\circle{32}}
\put(48,85){\large{$D$}}
\end{picture}
\caption{Shape of $O^\mhyphen_{\GL_N/\tilde{P}_s}$}
\label{figure:Ominus}
\end{figure}
\end{examplebox}

\begin{remark}
\label{remark:SchubCMNormaletc}
$X_{{\tilde{P}}}(w)$ is an irreducible (and reduced) variety of dimension 
equal to the length of $w$.
(Here we use that $w$ is the representative in $W^{\tilde{P}}$.)
It can be seen easily that under the natural projection $\GL_N/B_N \to 
\GL_N/{\tilde{P}}$, $X_{B_N}(w)$ maps birationally onto $X_{\tilde{P}}(w)$
for every $w \in W^{\tilde{P}}$.
It is known that Schubert varieties
are normal, Cohen-Macaulay and have rational singularities;
see, e.g.,~\cite[Section~3.4]{BrionKumarFrobSplitting05}.
\end{remark}

\subsection{Homogeneous bundles and representations}
\label{sec:homogvb}

Let $Q$ be a parabolic subgroup of $\GL_n$. We collect here
some results about homogeneous vector-bundles on $\GL_n/Q$. Most of these
results are well-known, but for some of them, we could not find a
reference, so we give a proof here for the sake of completeness. Online
notes of G. Ottaviani~\cite {OttavianiRationalHomogVarieties1995} and of
D.~Snow~\cite {SnowHomogVB} discuss the details of many of these results.

Let $L_Q$ and $U_Q$  be respectively the Levi subgroup and the unipotent 
radical of $Q$.
Let $E$ be a finite-dimensional vector-space on which $Q$ acts on the
right; the vector-spaces that we will encounter have natural right action.

\begin{definition}
\label{definition:assbundle}
Define $\GL_n \times^Q E :=
(\GL_n \times E)/\sim$ where $\sim$ is the equivalence relation $(g,e) \sim 
(gq,eq)$ for every $g \in \GL_n$, $q \in Q$ and $e \in E$. 
Then $\pi_E : \GL_n \times^Q E \to \GL_n/Q, (g,e)\mapsto gQ$, is a 
vector-bundle called the
\define{vector-bundle associated to} $E$ (and the principal $Q$-bundle 
$\GL_n \to \GL_n/Q$).  For $g\in \GL_n,e\in E$, we write $[g,e] \in \GL_n 
\times^Q E$ for the equivalence class of
$(g,e) \in \GL_n \times E$ under $\sim$. 
We say that a vector-bundle $\pi : \bfE \to \GL_n/Q$ is
\define{homogeneous} if  $\bfE$ has a $\GL_n$-action and $\pi$ is
$\GL_n$-equivariant, i.e, for every $y \in \bfE$, $\pi(g\cdot y) = g \cdot
\pi(y)$.  \end{definition}

In this section, we abbreviate  $\GL_n \times^Q E$ as $\widetilde E$.  It 
is known that $\bfE$ is homogeneous if and only if
$\bfE \simeq \widetilde E$ for some $Q$-module $E$. (If this is the case, 
then $E$ is the fibre of
$\bfE$ over the coset $Q$.)
A homogeneous bundle $\widetilde E$ is said to be \define{irreducible} 
(respectively \define{indecomposable},
\define{completely reducible}) if $E$ is a
{irreducible} (respectively {indecomposable}, {completely reducible}) 
$Q$-module.
It is known that $E$ is completely reducible if and only if $U_Q$ acts
trivially and that $E$ is irreducible if and only if additionally it is
irreducible as a representation of $L_Q$. 
See~\cite[Section~5]{SnowHomogVB}
or~\cite[Section~10]{OttavianiRationalHomogVarieties1995} for the details.

Let $\sigma : \GL_n/Q \to \widetilde E$ be a section of $\pi_E$.  Let $g
\in \GL_n$; write $[h,f] = \sigma(gQ)$.  There exists a unique $q \in Q$ 
such that $h=gq$. 
Let $e = fq^{-1}$.
Then $[g,e] = [h,f]$. If $[h,f'] = [h,f]$, then
$f'=f$, so the assignment $g \mapsto e$ defines a function $\phi : \GL_n 
\to
E$. This is $Q$-equivariant in the following sense:
\begin{equation}
\label{equation:EquivForSectionsOfTildeE}
\phi(gq) = \phi(g)q, \;\text{for every $q \in Q$ and $g \in \GL_n$}.
\end{equation}
Conversely, any such map defines a section of $\pi_E$. The set of sections
$\homology^0(\GL_n/Q, \widetilde E)$
of $\pi_E$ is a vector-space with $(\alpha \phi)(g) = \alpha(\phi(g))$ for
every $\alpha \in \complex$, $\phi$ a section of $\pi_E$ and $g\in \GL_n$.
It is finite-dimensional.

Note that $\GL_n$ acts on $\GL_n/Q$ by multiplication on the left; setting 
$h \cdot [g,e] = [hg,e]$ for $g,h \in \GL_n$ and $e \in E$, we extend this 
to
$\widetilde E$. We can also define a natural $\GL_n$-action on
$\homology^0(\GL_n/Q, \widetilde E)$
as follows. For any map $\phi : \GL_n \to E$, set
$h \circ \phi$ to be the map $g \mapsto \phi(h^{-1}g)$.
If $\phi$ satisfies~\eqref{equation:EquivForSectionsOfTildeE}, then for every
$q \in Q$ and $g \in \GL_n$, $(h \circ \phi)(gq) =  \phi(h^{-1}gq) =
(\phi(h^{-1}g))q = ((h \circ \phi)(g))q$, so
$h \circ \phi$ also satisfies~\eqref{equation:EquivForSectionsOfTildeE}.
The action of $\GL_n$ on the sections is on the left:
$(h_2h_1) \circ \phi  = [g \mapsto \phi(h_1^{-1}h_2^{-1}g)] = [g \mapsto 
(h_1 \circ \phi)(h_2^{-1}g)] = h_2 \circ (h_1 \circ \phi)$. In fact, 
$\homology^i(\GL_n/Q, \widetilde E)$ is a
$\GL_n$-module for every $i$.

Suppose now that $E$ is one-dimensional. Then $Q$ acts on $E$ by a
character $\lambda$; we denote the associated line bundle on $\GL_n/Q$
by $L_\lambda$. 

\begin{discussionbox}
\label{discussionbox:taut}
Let $Q = P_{\widehat{i_1, \ldots, i_t}}$,
with $1 \leq i_1 < \cdots < i_t \leq n-1$.
A weight $\lambda$ is said to be \define{$Q$-dominant} if 
when we write $\lambda = \sum_{i=1}^n a_i \omega_i$ in terms of the
fundamental weights $\omega_i$, we have,  
$a_i \geq 0$ for all $i \not \in \{i_1, \ldots, i_t\}$, or equivalently, 
the associated line bundle (defined above)
$L_\lambda$ on $Q/B_n$ has global sections.
If we express $\lambda$ as $\sum_{i=1}^n \lambda_i \epsilon_i$, 
then $\lambda$ is $Q$-dominant if and only if 
for every $0 \leq j \leq t$, $\lambda_{i_j+1} \geq \lambda_{i_j+2} \geq
\cdots \geq \lambda_{i_{j+1}}$ where we set $i_0 =0$ and $i_{r+1} = n$.
We will write $\lambda = (\lambda_1, \ldots, \lambda_n)$ 
to mean that $\lambda = \sum_{i=1}^n \lambda_i \epsilon_i$.
Every finite-dimensional irreducible $Q$-module is of the form 
$\homology^0(Q/B_n, L_\lambda)$ for a $Q$-dominant weight $\lambda$.
Hence the irreducible homogeneous vector-bundles on $\GL_n/Q$ are in correspondence with 
$Q$-dominant weights. We describe them now. If $Q = P_{\widehat{n-i}}$,
then $\GL_n/Q = {\Grass_{i,n}}$. (Recall that, for us, the $\GL_n$-action 
on
$\complex^n$ is on the right.)
On $\Grass_{i,n}$, we have the \define{tautological sequence}
\begin{equation}
\label{equation:tautseq}
0 \to \calR_i \to \complex^{n} \otimes \strSh_{\Grass_{i,n}} \to 
\calQ_{n-i} \to 0
\end{equation}
of homogeneous vector-bundles. The bundle $\calR_i$ is called the
\define{tautological sub-bundle} (of the trivial bundle $\complex^n$) and 
$\calQ_{n-i}$ is called the \define{tautological quotient bundle}.  Every 
irreducible homogeneous bundle on $\Grass_{i,n}$ is of the form 
$\Schur_{(\lambda_{1}, \cdots, \lambda_{n-i})} \calQ_{n-i}^* \otimes 
\Schur_{(\lambda_{n-i+1}, \cdots, \lambda_{n})} \calR_{i}^*$
for some $P_{\widehat{n-i}}$-dominant weight $\lambda$. 
Here $\Schur_{\mu}$ denotes the \define{Schur functor} associated to the
partition $\mu$.
Now suppose that $Q = P_{\widehat{i_1, \ldots, i_t}}$ with $1 \leq i_1 <
\cdots < i_t \leq n-1$.  Since the action is on the right, $\GL_n/Q$ 
projects to $\Grass_{n-i,n}$
precisely when $i = i_j$ for some $1 \leq j \leq t$. For each $1 \leq j 
\leq t$, we can take the pull-back of the tautological bundles
$\calR_{n-i_j}$ and
$\calQ_{i_j}$ to $\GL_n/Q$ from $\GL_n/P_{\widehat{i_j}}$.
The irreducible homogeneous bundle corresponding to a $Q$-dominant weight
$\lambda$ is
$\Schur_{(\lambda_1, \ldots, \lambda_{i_1})}\calU_{i_1} \otimes
\Schur_{(\lambda_{i_1+1}, \ldots,
\lambda_{i_2})}(\calR_{n-i_1}/\calR_{n-i_2})^*
\otimes \ldots \otimes
\Schur_{(\lambda_{i_{t-1}+1}, \ldots,
\lambda_{i_t})}(\calR_{n-i_{t-1}}/\calR_{n-i_t})^*
\otimes
\Schur_{(\lambda_{i_{t}+1}, \ldots,
\lambda_{i_n})}(\calR_{n-i_{t}})^*$.
See~\cite[Section~4.1]{WeymSyzygies03}.
Hereafter, we will write $\calU_i = \calQ_i^*$. Moreover, abusing
notation, we will use $\calR_{i}$, $\calQ_i$, $\calU_i$ etc. for
these vector-bundles on any (partial) flag varieties on which they would
make sense.
\end{discussionbox}

A $Q$-dominant weight is called \define{$(i_1, \ldots, i_r)$-dominant}
in~\cite[p.~114]{WeymSyzygies03}.  Although our definition looks like
Weyman's definition, we should keep in mind that our action is on the
right. We only have to be careful when we apply the Borel-Weil-Bott theorem
(more specifically, Bott's algorithm). In this paper, our computations are
done only on Grassmannians. If $\mu$ and $\nu$ are partitions, then 
$(\mu,\nu)$ will be $Q$-dominant (for a suitable $Q$), and will give us the
vector-bundle $\Schur_\mu \calQ^* \otimes \Schur_\nu \calR^*$ (this is
where the right-action of $Q$ becomes relevant) and to compute its
cohomology, we will have to apply Bott's algorithm to the $Q$-dominant
weight $(\nu,\mu)$.  (In~\cite{WeymSyzygies03}, one would get $\Schur_\mu 
\calR^* \otimes \Schur_\nu \calQ^*$ and would apply Bott's
algorithm to $(\mu,\nu)$.) See, for example, the proof of
Proposition~\ref{proposition:vanishingForxi} or the examples that follow 
it.

\begin{proposition}
\label{proposition:higherdirectimage}
Let $Q_1 \subseteq Q_2$ be parabolic subgroups and $E$ a $Q_1$-module.  Let 
$f : \GL_n/Q_1 \to \GL_n/Q_2$ be the natural map. Then for every $i \geq 0$, 
$R^if_*(\GL_n \times^{Q_1} E) = \GL_n \times^{Q_2} \homology^i(Q_2/Q_1,
\GL_n \times^{Q_1} E)$.
\end{proposition}

\begin{proof}
For $Q_2$ (respectively, $Q_1$), the category of homogeneous vector-bundles 
on
$\GL_n/Q_2$ (respectively, $\GL_n/Q_1$) is equivalent to the category of
finite-dimensional $Q_2$-modules (respectively, finite-dimensional
$Q_1$-modules). Now, the functor $f^*$ from the category of homogeneous
vector-bundles over $\GL_n/Q_2$ to that over $\GL_n/Q_1$ is equivalent to
the restriction functor $\mathrm{Res}_{Q_1}^{Q_2}$. Hence their
corresponding right-adjoint functors $f_*$ and the induction functor
$\mathrm{Ind}_{Q_1}^{Q_2}$
are equivalent; one may refer to~\cite[II.5, p.~110]{HartAG}
and~\cite[I.3.4, `Frobenius Reciprocity']{JantzenReprAlgGps2003} to see
that these are indeed adjoint pairs. Hence, for homogeneous bundles on
$\GL_n/Q_1$, $R^if_*$ can be computed using $R^i\mathrm{Ind}_{Q_1}^{Q_2}$.
On the other hand, note that $\mathrm{Ind}_{Q_1}^{Q_2}(-)$ is the functor
$\homology^0(Q_2/Q_1, \GL_n \times^{Q_1}-)$ on $Q_1$-modules (which follows 
from~\cite[I.3.3, Equation~(2)]{JantzenReprAlgGps2003}).  The proposition 
now follows.
\end{proof}

\section{Properties of Schubert desingularization}
\label{sec:schubertdesing}
\numberwithin{equation}{section}

This section is devoted to proving some results on smooth Schubert
varieties in partial flag varieties.
In Theorem~\ref{theorem:smoothnessAndLinearity}, we show that opposite
cells of certain smooth Schubert varieties in $\GL_N/{\tilde{P}}$ are linear
subvarieties of the affine variety $O^\mhyphen_{\GL_N/{\tilde{P}}}$, where
$\tilde{P} = \tilde{P}_s$ for some $1 \leq s \leq n-1$. 
Using this, we show in
Theorem~\ref{theorem:birational} that if 
$X_P(w) \in \GL_N/P$ is such that there exists a parabolic subgroup
$\tilde{P} \subsetneq P$ such that the birational model 
$X_{\tilde{P}}(\tilde{w}) \subseteq \GL_N/\tilde{P}$ of $X_P(w)$ is
smooth (we say that $X_P(w)$  has a \define{Schubert
desingularization} if this happens) then the inverse image of 
$Y_P(w)$ inside $X_{\tilde{P}}(\tilde{w})$ is a vector-bundle over a
Schubert variety in $P/\tilde{P}$. This will give us a realization of
Diagram~\eqref{equation:KLWdiagram}.

Recall the following result about the tangent space of a Schubert
variety; see~\cite[Chapter~4]{BilleyLakshmibaiSingularLoci2000} for
details.

\begin{proposition}\label{proposition:tgt}
Let $\tau\in W^{\tilde{P}}$. Then the dimension of the tangent space of 
$X_{\tilde{P}}(\tau)$ at $e_{\mathrm{id}}$ is 
\[
\#\{s_{\alpha} \mid \alpha \in R^\mhyphen\setminus R^\mhyphen_{\tilde{P}}
\;\text{and}\; \tau\ge s_{\alpha} \;\text{in}\; W/W_{\tilde{P}} \}.
\] 
In particular, $X_{\tilde{P}}(\tau)$ is smooth if and only if
$\dim X_{\tilde{P}}(\tau) = 
\#\{s_{\alpha} \mid \alpha \in R^\mhyphen\setminus R^\mhyphen_{\tilde{P}}
\;\text{and}\; \tau\ge s_{\alpha} \;\text{in}\; W/W_{\tilde{P}} \}$.
\end{proposition}

\begin{definition}
\label{definition:linear}
Say that a Schubert variety $X_{\tilde{P}}(\tau)$ \define{has \linearity} if 
$Y_{\tilde{P}}(\tau)$ (which is defined as $X_{\tilde{P}}(\tau) \cap
O^\mhyphen_{G/\tilde{P}}$) is a coordinate subspace of 
$O^\mhyphen_{G/\tilde{P}}$, defined by the vanishing of some of the
variables $x_{ij}$ from Discussion~\ref{discussionbox:Ominus}.
\end{definition}

It is immediate that if $X_{\tilde{P}}(\tau)$ has \linearity then it is smooth.
The converse is not true, as the following example shows. Let $\tau =
(2,4,1,3)$ and consider $X_B(\tau) \subseteq \GL_4/B$. Note that 
$X_B(\tau)$ is smooth. The reflections $(i,j)$ (with $i > j$) that satisfy
$(i,j) \not \leq \tau$ (in $W = S_4$) are precisely $(3,1)$, $(4,1)$ and
$(4,2)$. For these reflections, we note that the relevant restrictions of
the Pl\"ucker coordinates to ${O^\mhyphen_{\GL_4/B}}$ that vanish on
$Y_{\tilde{P}}(\tau)$ are as follows:
$p_{(3,1)}^{(1)}|_{O^\mhyphen_{\GL_4/B}} = x_{31}$,
$p_{(4,1)}^{(1)}|_{O^\mhyphen_{\GL_4/B}} = x_{41}$ and
$p_{(4,2)}^{(3)}|_{O^\mhyphen_{\GL_4/B}} = x_{32}x_{43}-x_{42}$. Hence 
$Y_{\tilde{P}}(\tau)$ is defined by $x_{31} = 0,x_{41} =0, x_{32}x_{43}-x_{42}
= 0$ as a sub variety of  ${O^\mhyphen_{\GL_4/B}}$, showing that
$Y_{\tilde{P}}(\tau)$ is a smooth subvariety of ${O^\mhyphen_{\GL_4/B}}$ but
not a coordinate subspace.

We are interested in the parabolic subgroups 
$\tilde{P} = \tilde{P}_s$ for some $1 \leq s \leq n-1$. 
Take such a $\tilde{P}$. We will show below that certain
smooth Schubert varieties in $\GL_N/\tilde{P}$ have \linearity.
From Discussion~\ref{discussionbox:Ominus} it follows that 
$\{x_{ij}  \mid j \leq n \;\text{and}\; i \geq \max\{j+1,s+1\}\}$
is a system of affine coordinates for $O^\mhyphen_{\GL_N/{\tilde{P}}}$. 

\begin{notationbox}\label{note}
For the remainder of this section
we adopt the following notation: 
Let $w = (a_1, a_2, \ldots, a_n) \in W^P$. 
Let $r = r(w)$, i.e., the index $r$ such that $a_r \leq n < a_{r+1}$.
Let $1 \leq s \leq r$.
We write $\tilde{P} = \tilde{P}_s$.
Let $\tilde{w}$ be the minimal representative of $w$ in $W^{\tilde{P}}$.
Let $c_{r+1} > \cdots > c_n$ be such that
$\{c_{r+1}, \ldots, c_n\} = \{1, \ldots, n\} \setminus \{a_{1}, \ldots,
a_{r}\}$; let $w' := (a_1, \ldots, a_r, c_{r+1}, \ldots, c_n) \in S_n$, the 
Weyl group of $\GL_n$.
\end{notationbox}

\begin{theorem}
\label{theorem:smoothnessAndLinearity}
With notation as above, suppose that the Schubert variety
$X_{\tilde{P}}(\tilde{w})$ of $\GL_N/\tilde{P}$ is smooth. Then it has
\linearity.
\end{theorem}

\begin{proof}
Since $a_1 < \cdots < a_n$, we see that for every $j \leq n$ and for every
$i \geq \max\{a_j+1, s+1\}$, the reflection $(i,j)$ equals
$(1,2,\ldots, j-1, i)$ in 
$W/W_{P_{\widehat{j}}}$, while $\tilde{w}$ equals $(a_1, \ldots, a_j)$. 
Hence $(i,j)$ is not smaller than
$\tilde{w}$ in $W/W_{P_{\widehat{j}}}$, so the Pl\"ucker coordinate
$p_{{(i,j)}}^{(j)}$ vanishes on $X_{\tilde{P}}(\tilde w)$. Therefore for such
$(i,j)$, $x_{ij} \equiv 0$ on $Y_{\tilde{P}}(\tilde w)$, by
\eqref{equation:minorcoordinate}. 

On the other hand, note that the reflections 
$(i,j)$ with $j \leq n$ and $i \geq \max\{a_j+1, s+1\}$ are precisely the
reflections $s_{\alpha}$ with 
$\alpha \in R^\mhyphen\setminus R^\mhyphen_{\tilde{P}}$ 
and $\tilde{w}\not\ge s_{\alpha}$ in $W/W_{\tilde{P}}$.
Since $X_{\tilde{P}}(\tilde w)$ is smooth, this implies
(see Proposition~\ref{proposition:tgt}) 
that the codimension of 
$Y_{\tilde{P}}(\tilde w)$ in $O^\mhyphen_{\GL_N/{\tilde{P}}}$ equals 
\[
\#\left\{(i,j) \mid j \leq n \;\text{and}\; i \geq \max\{a_j+1, s+1\}\right\}
\]
so 
$Y_{\tilde{P}}(\tilde w)$ is the linear subspace of
$O^\mhyphen_{\GL_N/{\tilde{P}}}$ defined by the vanishing of 
$\{x_{ij} \mid j \leq n \;\text{and}\; i \geq \max\{a_j+1, s+1\}\}$.
\end{proof}

We have the following immediate corollary to the proof of
Theorem~\ref{theorem:smoothnessAndLinearity}.
See Figure~\ref{figure:Ominus} for a picture.

\begin{corollary}\label{corollary:iden}
Suppose that $X_{\tilde{P}}(\tilde w)$ is smooth, and identify 
$O^-_{G/{\tilde{P}}}$ with $O^-_{G/P}\times O^-_{P/{\tilde{P}}}$ (as in Figure \ref{figure:Ominus}).
Then we have an identification of
$Y_{\tilde{P}}(\tilde w)$ with
$\calV_w \times \calV'_w$, where $\calV_w$ is the linear subspace of  
$O^-_{G/{P}}$ (note that $O^-_{G/{P}}$ is identified with $M_{m,n}$, the space of all $m\times n$ matrices)
given by 
\[
x_{ij} = 0 \;\text{if}\;
\begin{cases}
1 \leq j \leq r(w) \;\text{and for every}\; i, \\
or, \\
r(w)+1 \leq j \leq n-1 \;\text{and}\; a_j -n < i \leq m. \\
\end{cases}
\]
and $\calV'_w$ is the linear subspace 
of $O^-_{P/{\tilde{P}}}$ (being identified with $M_{m,n}$, the space of all $m\times n$ matrices)given by 
\[
x_{ij}=0 \;\text{for every}\;  1\le j\le r(w) \;\text{and for every}\; i \geq 
\max \{a_j+1, s+1\}.
\]
\end{corollary}

\begin{proof} 
As seen in the proof of Theorem~\ref{theorem:smoothnessAndLinearity}, we
have  that $Y_{\tilde{P}}(\tilde w)$ is the subspace of the affine space
$O^-_{G/{\tilde{P}}}$ given by
$x_{ij}=0$ for every $j \leq n$ and for every $i \geq \max\{a_j+1,s+1\}$.
This fact together with the identification of $O^-_{G/{\tilde{P}}}$ with
$O^-_{G/P}\times O^-_{P/{\tilde{P}}}$, implies that
we have an identification of $Y_{\tilde{P}}(\tilde w)$ (as a sub variety of
$O^-_{G/P}\times O^-_{P/{\tilde{P}}}$) with
$\calV_w \times \calV'_w$.
\end{proof}

Let $Z_{\tilde{P}}(\tilde w) := 
Y_P(w) \times_{X_P(w)} X_{\tilde{P}}(\tilde w)
= (O^\mhyphen_{\GL_N/P} \times P/{\tilde{P}}) \cap X_{{\tilde P}}(\tilde
w)$.

Write $p$ for the composite
map $Z_{\tilde{P}}(\tilde w) \to O^\mhyphen_{\GL_N/P} \times P/{\tilde{P}}
\to P/{\tilde{P}}$, where the first map is the inclusion (as a closed
subvariety)
and the second map is projection.
Using Proposition~\ref{proposition:structureOMinus}%
\eqref{proposition:structureOMinusInverseImage}
and~\eqref{proposition:SLnModQ} we see that 
\[
p \left(
{\begin{bmatrix}
A_{n \times n} & 0_{n \times m} \\
D_{m \times n} & I_{m}
\end{bmatrix}} (\mathrm{mod}\, {\tilde{P}})\right)
=A (\mathrm{mod}\,Q_s).
\]
($A$ is invertible by
Proposition~\ref{proposition:structureOMinus}%
\eqref{proposition:structureOMinusinvertibleA}.)

Using the injective map
\[
A \in B_n \mapsto {\begin{bmatrix}
A & 0_{n \times m} \\
0_{m \times n} & I_{m}
\end{bmatrix}} \in B_N,
\]
$B_n$ can be thought of as a subgroup of $B_N$. 
With this identification,  we have the following Proposition:
\begin{proposition}
\label{proposition:stable} 
$Z_{\tilde P}(\tilde w)$ is $B_n$-stable (for
the action on the left by multiplication). Further, $p$ is
$B_n$-equivariant.
\end{proposition}
\begin{proof} Let
\[
z := {\begin{bmatrix}
A_{n \times n} & 0_{n \times m} \\
D_{m \times n} & I_{m}
\end{bmatrix}} \in \GL_N
\]
be such that $z \tilde{P} \in Z_{\tilde P}(\tilde w)$. Since
$X_{B_N}(\tilde{w}) \to X_{\tilde P}(\tilde{w})$ is surjective, we may
assume
that $z(\mathrm{mod}\, B_N) \in X_{B_N}(\tilde w)$, i.e., 
$z \in \overline{B_N \tilde{w} B_N}$.
Then for every $A' \in B_n$
\[
{\begin{bmatrix}
A' & 0_{n \times m} \\
0_{m \times n} & I_{m}
\end{bmatrix}}z = {\begin{bmatrix}
A'A & 0 \\
D & I_{m}
\end{bmatrix}} =: z'.
\]
Then $z' \in \overline{B_N \tilde{w} B_N}$, so
$z' (\mathrm{mod}\, \tilde{P}) \in X_{\tilde P}(\tilde w)$. 
By
Proposition~\ref{proposition:structureOMinus}\eqref%
{proposition:structureOMinusinvertibleA}, we have that $A$ is invertible,
and hence $AA'$ is invertible;  this implies (again by
Proposition~\ref{proposition:structureOMinus}\eqref%
{proposition:structureOMinusinvertibleA}
) that 
$z' (\mathrm{mod}\, \tilde{P}) \in Z_{\tilde P}(\tilde w)$. 
Thus $Z_{\tilde P}(\tilde w)$ is $B_n$-stable.
Also,  $p(A'z)=p(z') = A'A = A'p(z)$. Hence $p$ is $B_n$-equivariant.
\end{proof}

\begin{theorem}
\label{theorem:birational}
With notation as above,
\begin{enumerate}

\item \label{enum:birationalMap}
The natural map $X_{\tilde P}({\tilde{w}}) \to X_P(w)$
is proper and birational. In particular, the map $Z_{\tilde{P}}(\tilde w)
\to Y_P(w)$ is proper and birational.

\item \label{enum:birationalFibre}
$X_{Q_s}(w')$ is the fibre of the natural map
$Z_{\tilde P}({\tilde{w}}) \to Y_P(w)$ at $e_{id} \in  Y_P(w)$  ($w'$ being as in Notation \ref{note}).

\item \label{enum:birationalFibration}
Suppose that $X_{\tilde P}({\tilde{w}})$ is smooth. Then
$X_{Q_s}(w')$ is the image of $p$. Further, $p$ is a fibration with fibre
isomorphic to $\calV_w$.

\item \label{enum:birationalVB}
Let $X_{\tilde P}({\tilde{w}})$ be smooth. Then $p$ identifies
$Z_{\tilde{P}}(\tilde w)$ as a sub-bundle of  the trivial bundle
$O^\mhyphen_{\GL_N/P} 
\times X_{Q_s}(w')$, which arises as the
restriction of the vector-bundle on $\GL_n/Q_s$ associated to the
$Q_s$-module $\calV_w$ (which, in turn, is a $Q_s$-submodule of
$O^\mhyphen_{\GL_N/P}$).
\end{enumerate}
\end{theorem}

We believe that all the assertions above hold without the hypothesis that 
$X_{\tilde P}({\tilde{w}})$ is smooth.

\begin{proof}
\eqref{enum:birationalMap}: The map $X_{\tilde{P}}(\tilde w) 
\hookrightarrow \GL_N/{\tilde{P}} \to
\GL_N/P$ is proper and its (scheme-theoretic) image is $X_P(w)$; hence
$X_{\tilde{P}}(\tilde w) \to X_P(w)$ is proper. Birationality follows 
from
the fact that $\tilde w$ is the minimal representative of the coset
$w\tilde P$ (see Remark~\ref{remark:SchubCMNormaletc}).

\eqref{enum:birationalFibre}: 
The fibre at $e_{\mathrm{id}} \in  Y_P(w)$ of the map 
$Y_{\tilde{P}}(\tilde w) \to Y_P(w)$ is 
$\{0\} \times \calV'_w (\subseteq \calV_w \times \calV'_w =
Y_{\tilde{P}}(\tilde w))$.
Since $Z_{\tilde{P}}(\tilde w)$ is the closure of $Y_{\tilde{P}}(\tilde w)$
inside $O^\mhyphen_{\GL_N/P} \times P/\tilde{P}$
and $X_{Q_s}(w')$ is the closure of $\calV'_w$ inside $P/\tilde{P}$ (note
that as a sub variety of $O^\mhyphen_{P/\tilde{P}}$, 
$Y_{Q_s}(w')$ is identified with   $\calV'_w$ ), we see
that fibre of $Z_{\tilde P}({\tilde{w}}) \to Y_P(w)$ at $e_{id} \in
Y_P(w)$ is $X_{Q_s}(w')$.

\eqref{enum:birationalFibration} From
Theorem~\ref{theorem:smoothnessAndLinearity} it follows that
\[
Y_{\tilde{P}}(\tilde w) = 
\left\{
\begin{bmatrix} A_{n \times n} & 0_{n \times m} \\ D_{m \times n} & I_m
\end{bmatrix}\mod \tilde P \mid A \in \calV'_w \;\text{and}\;
D \in \calV_w \right\}.
\]
Hence $p(Y_{\tilde{P}}(\tilde w)) = 
\calV'_w \subseteq {X_{Q_s}(w')}$.
Since $Y_{\tilde{P}}(\tilde w)$ is dense inside
$Z_{\tilde{P}}(\tilde w)$ and 
${X_{Q_s}(w')}$ is closed in $\GL_n/Q_s$, 
we see that $p(Z_{\tilde{P}_r}(\tilde w)) \subseteq {X_{Q_s}(w')}$. The
other inclusion  $ {X_{Q_s}(w')}\subseteq p(Z_{\tilde{P}_r}(\tilde w))$
follows from~\eqref{enum:birationalFibre}. Hence, $p(Z_{\tilde{P}_r}(\tilde
w))$  equals ${X_{Q_s}(w')}$.

Next, to prove the second assertion in  \eqref{enum:birationalFibration},
we shall show that for every $A\in \GL_{n}$ with $A \mod Q_s \in
X_{Q_s}(w')$,
\begin{equation}
\label{equation:fibreOfp}
p^{-1}(A \mod Q_s) = \left\{
{\begin{bmatrix}
A & 0_{n \times m} \\
D & I_{m}
\end{bmatrix}} \mod \tilde P: D \in \calV_w \right\}.
\end{equation}

Towards proving this, we first observe that $p^{-1}(e_{id})$ equals
$\calV_w$ (in view of Corollary~\ref{corollary:iden}). Next, we observe
that every $B_n$-orbit inside ${X_{Q_s}(w')}$ meets
$\calV'_w(=Y_{Q_s}(w'))$; further, $p$ is $B_n$-equivariant (see
Proposition~\ref{proposition:stable}). The
assertion~\eqref{equation:fibreOfp} now follows. 
 
\eqref{enum:birationalVB}: First observe that for the action of right
multiplication by $GL_n$ on $\mathcal{O}^-_{G/P}$ (being identified with
$M_{m,n}$, the space of $m\times n$ matrices), $\calV_w$ is stable; we thus
get the homogeneous bundle $\GL_{n}\times^{Q_s} \calV_w\rightarrow
\GL_n/Q_s$ (Definition~\ref{definition:assbundle}). Now to prove the
assertion about
$Z_{\tilde{P}_s}(\tilde w))$ being a
vector-bundle over ${X_{Q_s}(w')}$, we will show that there is a 
commutative
diagram given as below, with $\psi$ an isomorphism:
\[
\xymatrix{%
Z_{\tilde{P}_s}(\tilde w) \ar[rrrd]^{\phi} \ar[rrdd]_p \ar[rrd]_\psi\\
&&(\GL_{n}\times^{Q_s} \calV_w)|_{X_{Q_s}(w')} \ar[d] \ar[r]&  
\GL_{n}\times^{Q_s} \calV_w \ar[d]^\alpha \\
&&{X_{Q_s}(w')} \ar[r]^\beta& \GL_n/Q_s
}
\]
The map $\alpha$ is the homogeneous bundle map
and $\beta$ is the inclusion map.  
Define $\phi$ by \[
\phi :  {{\begin{bmatrix}
A & 0_{n \times m} \\
D & I_{m}
\end{bmatrix}} \mod \tilde P}
 \longmapsto (A,D)/\sim.
\]
Using Proposition~\ref{proposition:structureOMinus}\eqref%
{proposition:structureOMinusInverseImage} and
\eqref{equation:fibreOfp}, we conclude the following:
$\phi$ is well-defined and injective;
$\beta \cdot p = \alpha \cdot \phi$; hence, by the universal property of 
products, the map $\psi$ exists; and, finally, the injective map $\psi$ is
in fact an isomorphism (by dimension considerations).
\end{proof}

\begin{corollary}
\label{corollary:ourRealization}
If $X_{\tilde{P}}(\tilde{w})$ is smooth, then we have the following
realization of the diagram in~\eqref{equation:KLWdiagram}:
\[
\vcenter{\vbox{%
\xymatrix{%
Z_{\tilde P}(\tilde w) \ar[d]^{q'} \xyhookrightarrow[r] &
O^\mhyphen_{\GL_N/P} \times 
X_{Q_s}(w')  \ar[d]^{q} \ar[r] & X_{Q_s}(w')\\
Y_{P}(w)  \xyhookrightarrow[r] & O^\mhyphen_{\GL_N/P}
}}}.
\]
\end{corollary}

We now describe a class of smooth varieties
$X_{\tilde P_s}({\tilde{w}})$ inside $\GL_N/\tilde{P_s}$.

\begin{proposition}
\label{proposition:smoothXtildew}
$X_{\tilde P_s}({\tilde{w}})$ is smooth in the following situations:
\begin{enumerate}

\item \label{enum:smoothXtildewkempfDesing}
$w \in W^P$ arbitrary and $s=1$~\cite{KempfSchubertMethodsCurves1971}. 

\item \label{enum:smoothXtildewOurDesing}
$w = (n-r+1, \ldots, n, a_{r+1}, \cdots, a_{n-1}, N) \in W^P$ for
some $1 \leq r \leq n-1$ and $s=r$.

\end{enumerate}
\end{proposition}

\begin{proof}
For both \eqref{enum:smoothXtildewkempfDesing} and
\eqref{enum:smoothXtildewOurDesing}: 
Let $w_{\mathrm{max}} \in W (= S_N)$ be the maximal representative of
$\tilde{w}$.
We {\textbf {claim}} that 
\[
w_{\mathrm{max}}  = (a_s, a_{s-1}, \ldots, a_1, a_{s+1}, a_{s+2},
\ldots, a_n, b_{n+1}, \ldots, b_N) \in W.
\]
Assume the claim. Then $w_{\mathrm{max}}$ is a
$4231$- and $3412$-avoiding element of $W$; hence 
$X_{B_N}(w_{\mathrm{max}})$ is smooth (see ~\cite 
{LakshmiSandhyaSmoothnessSLnModB1990} , ~\cite[8.1.1]{BilleyLakshmibaiSingularLoci2000}). 
Since $w_{\mathrm{max}}$ is the maximal representative (in $W$) of 
$\tilde w{\tilde{P}_s}$,
we see that $X_{B_N}(w_{\mathrm{max}})$ is a fibration over
$X_{\tilde{P}_s}(\tilde w)$ with smooth fibres 
${\tilde{P}_s}/B_N$; therefore 
$X_{\tilde{P}_s}(\tilde w)$ is smooth.

To prove the claim, we need to show that 
$X_{P_{\widehat{i}}}(w_{\mathrm{max}}) = X_{P_{\widehat{i}}}(\tilde w)$
for every $s
\leq i \leq n$ and that $w_{\mathrm{max}}$ is the maximal element of $W$
with this property. This follows, since for every $\tau := (c_1, \ldots, c_N)
\in W$ and for every $1 \leq i \leq N$, $X_{P_{\widehat{i}}}(\tau) = 
X_{P_{\widehat{i}}}(\tau')$ where $\tau' \in W^{P_{\widehat{i}}}$ is the
element with $c_1, \ldots, c_i$ written in the increasing order.
\end{proof}

In light of
Proposition~\ref{proposition:smoothXtildew}\eqref{enum:smoothXtildewOurDesing}
we make the following definition. Our concrete descriptions of free
resolutions will be for this class of Schubert varieties.
\begin{definition}
\label{definition:ClassWk}
Let $1 \leq r \leq n-1$. 
Let
$\calW_r = \{(n-r+1, \ldots, n, a_{r+1}, \cdots, a_{n-1}, N)  \in W^P : n < 
a_{r+1} < \cdots < a_{n-1} < N\}$.
\end{definition}

The determinantal variety of $(m \times n)$ matrices of rank at most $k$
can be realized as $Y_P(w)$, $w = (k+1, \ldots, n, N-k+1, \ldots N) \in
\calW_{n-k}$~\cite[Section~1.6]{SeshadriSMT07}.

\begin{example}
This example shows that even with $r=s$, $X_{Q_s}(w')$ need not be smooth
for arbitrary $w \in W^P$.
Let $n=m=4$ and $w = (2,4,7,8)$. Then $r=2$; take $s=2$. Then we obtain 
$w_{\mathrm{max}} = (4,2,7,8,5,6,3,1)$, which has a $4231$ pattern.
\end{example}

\section{Free resolutions}
\label{sec:freeresolutions}

\subsection*{Kempf-Lascoux-Weyman geometric technique}%

We summarize the geometric technique of computing free resolutions,
following \cite[Chapter~5]{WeymSyzygies03}.  Consider 
Diagram~\eqref{equation:genericKLW}.
There is a natural map $f : V \to \Grass_{r, d}$ (where $r = \rank_V\! Z$ 
and $d=\dim \bbA$)
such that the inclusion $Z \subseteq \bbA \times V$ is
the pull-back of the tautological sequence~\eqref{equation:tautseq};
here $\rank_V Z$ denotes the rank of $Z$ as a vector-bundle  over $V$, 
i.e., $\rank_V Z = \dim Z - \dim V$.
Let $\xi = (f^* \calQ)^*$. Write $R$ for the polynomial ring
$\complex[\bbA]$ and $\frakm$ for its homogeneous maximal ideal.  (The 
grading on $R$ arises as follows. In
Diagram~\eqref{equation:genericKLW}, $\bbA$ is thought of as the fibre of a
trivial vector-bundle, so it has a distinguished point, its origin. Now,
being a sub-bundle, $Z$ is defined by linear equations in each fibre; i.e.,
for each $v \in V$, there exist $s := (\dim \bbA - \rank_V Z)$ linearly
independent linear polynomials $\ell_{v, 1}, \ldots, \ell_{v, s}$ that 
vanish along $Z$ and define it. Now $Y = \{y \in \bbA :
\text{there exists $v \in V$ such that}\; \ell_{v,1}(y) = \cdots =
\ell_{v,s}(y)=0\}$. Hence $Y$ is defined by homogeneous polynomials. This
explains why the resolution obtained below is graded.) Let $\frakm$ be the
homogeneous maximal ideal, i.e., the ideal defining the origin in $\bbA$.
Then:
\begin{theorem}[\protect{\cite[Basic Theorem~5.1.2]{WeymSyzygies03}}]
\label{theorem:geometrictechnique}
With notation as above, there is a finite complex $(F_\bullet,
\partial_\bullet)$ of finitely generated graded free $R$-modules that is
quasi-isomorphic to $\RDer q'_* \strSh_{Z}$, with \[
F_i = \oplus_{j \geq 0} \homology^j(V, \bigwedge^{i+j} \xi)
\otimes_\complex R(-i-j),
\]
and $\partial_i(F_i) \subseteq \frakm F_{i-1}$.  Furthermore, the following
are equivalent:
\begin{asparaenum}
\item $Y$ has rational singularities, i.e.  $\RDer q'_* \strSh_{Z}$ is 
quasi-isomorphic to $\strSh_{Y}$;
\item $F_\bullet$ is a minimal $R$-free resolution of $\complex[Y]$, i.e., 
$F_0 \simeq R$ and $F_{-i} = 0$ for every $i > 0$.
\end{asparaenum}
\end{theorem}

We give a sketch of the proof because one direction of the
equivalence is only implicit in the proof of~\cite[5.1.3]{WeymSyzygies03}.

\begin{proof}[Sketch of the proof] One constructs a suitable $q_*$-acyclic
resolution $\calI^\bullet$ of the Koszul complex that resolves $\strSh_{Z}$ 
as an $\strSh_{\bbA \times V}$-module so that the
terms in $q_*\calI^\bullet$ are \emph{finitely generated free} graded
$R$-modules. One places the Koszul complex on the negative horizontal axis
and thinks of $\calI^\bullet$ as a second-quadrant double complex, thus to
obtain a complex $G_\bullet$ of finitely generated free $R$-modules whose
homology at the $i$th position is $R^{-i}q_* \strSh_Z$. Then,
using standard homological considerations, one constructs a subcomplex
$(F_\bullet, \partial_\bullet)$ of $G_\bullet$ that is quasi-isomorphic to
$G_\bullet$ with $\partial_i(F_i) \subseteq \frakm F_{i-1}$ (we say that
$F_\bullet$ is \define{minimal} if this happens), and since
$\homology_i(G_\bullet) = 0$ for every $|i| \gg 0$, $F_i = 0$ for every $|i| 
\gg 0$. Now using the minimality of $F_\bullet$, we see that
$R^iq_*\strSh_Z = 0$ for every $i \geq 1$ if and only if $F_{-i} = 0$ for every 
$i \geq 1$. When one of these  conditions
holds, then $F_\bullet$ becomes a minimal free resolution of
$q_*\strSh_Z$ which is a finitely generated
$\strSh_Y$-module, and therefore $q_*\strSh_Z = \strSh_Y$ if and only if 
$q_*\strSh_Z$ is generated by one element as an
$\strSh_Y$-module if and only if $q_*\strSh_Z$ is a generated by one 
element as an $R$-module
if and only if $F_0$ is a free $R$-module of rank one
if and only if $F_0 = R(0)$ since $\homology^0(V, \bigwedge^0 \xi) \otimes
R$ is a summand of $F_0$.
\end{proof}

\subsection*{Our situation}

We now apply Theorem~\ref{theorem:geometrictechnique} to our situation.
We keep the notation of Theorem~\ref{theorem:birational}.
Theorem~\ref{theorem:geometrictechnique} and
Corollary~\ref{corollary:ourRealization} yield the following result:

\begin{theorem}
\label{theorem:stepone}
Suppose that $X_{\tilde P_s}({\tilde{w}})$ is smooth.  Write $\calU_w$ for 
the restriction to $X_{Q_s}(w')$ of the vector-bundle on $\GL_n/Q_s$ 
associated to the $Q_s$-module $\left(O^\mhyphen_{\GL_N/P}/ \calV_w\right)^*$.
(This is the dual of the quotient of $O^\mhyphen_{\GL_N/P} \times X_{Q_s}(w')$ by 
$Z_{\tilde{P}_s}(\tilde w)$.)
Then we have a minimal $R$-free resolution $(F_\bullet, \partial_\bullet)$ 
of
$\complex[Y_P(w)]$  with
\[
F_i = \oplus_{j \geq 0} \homology^j(X_{Q_s}(w'), \bigwedge^{i+j} \calU_w)
\otimes_\complex R(-i-j).
\]
\end{theorem}

In the first case, $Q_s = B_n$, so $p$ makes $Z_{\tilde{P}_1}(\tilde w)$ a 
vector-bundle on a smooth Schubert subvariety
$X_{B_1}(w')$ of $\GL_n/B_n$. In the second case, $w'$ is
the maximal word in $S_n$, so $X_{Q_r}(w') = \GL_n/Q_r$; see
Discussion~\ref{discussionbox:AssocModToVw} for further details.

Computing the cohomology groups required in Theorem~\ref{theorem:stepone} 
in the general situation of Kempf's desingularization
(Proposition~\ref{proposition:smoothXtildew}%
\eqref{enum:smoothXtildewkempfDesing})
is a difficult problem, even though the relevant Schubert variety 
$X_{B_n}(w')$ is smooth.
Hence we are forced to restrict our attention to the subset of $W^P$
considered in Proposition~\ref{proposition:smoothXtildew}%
\eqref{enum:smoothXtildewOurDesing}.

The stipulation that, for $w \in \calW_r$,
$w$ sends $n$ to $N$ is not very
restrictive.
This can be seen in two (related) ways.  Suppose that $w$ does
not send $n$ to $N$. Then, firstly, $X_{P}(w)$ can be thought of as a
Schubert subvariety of a smaller Grassmannian.
Or, secondly, $\calU_w$ will contain the trivial bundle $\calU_n$ as a
summand, so $\homology^0(\GL_n/Q_r, \xi) \neq 0$, i.e., $R(-1)$ is a 
summand
of $F_1$. In other words, the defining ideal of $Y_{P}(w)$ contains a
linear form.

\begin{discussionbox}
\label{discussionbox:AssocModToVw}
We give some more details of the situation in
Proposition~\ref{proposition:smoothXtildew}\eqref{enum:smoothXtildewOurDesing}
that will be used in the next section.
Let $w  = (n-r+1, n-r+2, \ldots, n, a_{r+1}, \ldots, a_{n-1}, N) \in 
\calW_r$.
The space of $(m\times n)$ matrices is a $\GL_{n}$-module with a right
action; the subspace $\calV_w$ is $Q_r$-stable under this action. Thus
$\calV_w$ is a $Q_r$-module, and gives an associated vector-bundle
$(\GL_{n}\times^{Q_r} \calV_w)$ on $\GL_n/Q_r$. The action on the right
of $\GL_n$ on the space of $(m\times n)$ matrices breaks by rows; each row
is a natural $n$-dimensional representation of $\GL_n$. For each $1 \leq j
\leq m$, there is a unique $r \leq i_j \leq n-1$ such that $a_{i_j} < j+n 
\leq
a_{i_j+1}$.  (Note that $a_r=n$ and $a_n=N$.)
In row $j$, $\calV_w$ has rank $n-i_j$, and is a sub-bundle of
the natural representation. Hence the vector-bundle associated to the $j$th
row of $\calV_w$ is the pull-back of the  tautological
sub-bundle (of rank $(n-i_j)$)
on $\Grass_{n-i_j, n}$.
We denote this by $\calR_{n-i_j}$. Therefore $(\GL_{n}\times^{Q_r} 
\calV_w)$ is the vector-bundle $\calR_w :=
\bigoplus_{j=1}^m \calR_{n-i_j}$. 
Let $\calQ_w := \bigoplus_{j=1}^m \calQ_{i_j}$ where $\calQ_{i_j}$ the 
tautological quotient bundles corresponding to $\calR_{n-i_j}$.  
Then the vector-bundle $\calU_w$ on $\GL_n/Q_r$ that was 
defined in Theorem~\ref{theorem:stepone} is $\calQ_w^*$.
\end{discussionbox}

\section{Cohomology of Homogeneous Vector-Bundles}
\label{sec:steptwo}

It is, in general, difficult to compute the cohomology groups
$\homology^j(\GL_n/Q_r, \bigwedge^{t} \calU_w)$ in 
Theorem~\ref{theorem:stepone} for arbitrary $w \in \calW_r$.
In this section, we will discuss some approaches. We believe that this is a
problem of independent interest. Our method involves replacing $Q_r$
inductively by increasingly bigger parabolic subgroups, so we give the
general set-up below.

\begin{setup}
\label{setup:compute}
Let $1 \leq r\leq n-1$. Let $m_r, \ldots, m_{n-1}$ be non-negative integers
such that $m_r+ \cdots+ m_{n-1} = m$. Let $Q$ be a parabolic subgroup of
$\GL_n$ such that $Q \subseteq P_{\widehat i}$ for every $r \leq i \leq n-1$ 
such
that $m_i > 0$. We consider the homogeneous vector-bundle $\xi = 
\oplus_{i=r}^{n-1} \calU_i^{m_i}$ on $\GL_n/Q$, 
We want to compute the vector-spaces $\homology^j(\GL_n/Q_r, \bigwedge^{t} 
\xi)$.
\end{setup}

\begin{lemma}
\label{lemma:pullback}
Let $f : X' \to X$ be a fibration with fibre some Schubert subvariety
$Y$ of some (partial) flag variety.
Then $f_*\strSh_{X'} = \strSh_X$ and $R^if_*\strSh_{X'} = 0$ for every $i
\geq 1$.
In particular, for every locally free coherent sheaves $L$ on $X$, 
$\homology^i(X', f^*L) = \homology^i(X,L)$ for every $i \geq 0$.
\end{lemma}

\begin{proof}
The first assertion is a consequence of Grauert's
theorem~\cite[III.12.9]{HartAG} and the fact (see, for
example,~\cite[Theorem~3.2.1]{SeshadriSMT07})
that \[
\homology^i(Y, \strSh_Y) = \begin{cases}
\complex, & \text{if}\; i=0\\
0, & \text{otherwise}.
\end{cases}
\]
The second assertion follows from the projection formula and the Leray
spectral sequence.
\end{proof}

\begin{proposition}
\label{proposition:removeUnnecParabolics}
Let $m_i, r \leq i \leq n-1$ be as in Set-up~\ref{setup:compute}.  Let $Q' 
= \bigcap\limits_{\substack{r \leq i \leq n-1 \\ m_i > 0}}
P_{\hat i}$. Then
$\homology^*(\GL_n/Q, \bigwedge^t\xi) = \homology^*(\GL_n/Q', 
\bigwedge^t\xi)$ for every $t$.
\end{proposition}

\begin{proof}
The assertion follows from Lemma~\ref{lemma:pullback}, noting that 
$\bigwedge^t\xi$ on $\GL_n/Q$ is the pull-back of $\bigwedge^t\xi$ on 
$\GL_n/Q'$, under the natural morphism $\GL_n/Q \to \GL_n/Q'$.
\end{proof}

\begin{proposition}
\label{proposition:vanishingForxi}
For all $j$, $\homology^j(\GL_n/Q, \xi)=0$.
\end{proposition}

\begin{proof}
We want to show that $\homology^j(\GL_n/Q, \calU_i)=0$ for every $r \leq i 
\leq n-1$ and for every
$j$.
By Lemma~\ref{lemma:pullback} (and keeping
Discussion~\ref{discussionbox:taut} in mind),
it suffices to show that $\homology^j(\Grass_{n-i,n},\calU_i) = 0$ for every 
$r \leq i \leq n-1$ and for every $j$. To this end, we apply the
Bott's algorithm~\cite[(4.1.5)]{WeymSyzygies03} to the weight \[
\alpha := (\underbrace{0, \ldots, 0}_{n-i}, 1, \underbrace{0, \ldots, 
0}_{i-1}).
\]
Note that there is a permutation $\sigma$ such that $\sigma \cdot \alpha =
\alpha$. The proposition now follows.
\end{proof}

\subsection*{An inductive approach}
We are looking for a way to compute $\homology^*(\GL_n/Q, \bigwedge^t \xi)$
for a homogeneous bundle \[
\xi = \bigoplus_{i \in A} \calU_i^{\oplus_{m_i}}
\]
where $A \subseteq \{r, \ldots, n-1\}$ and $m_i > 0$ for every $i \in A$.
Using Proposition~\ref{proposition:removeUnnecParabolics}, we assume that 
$Q = P_{\widehat{A}}$.
(Using Proposition~\ref{proposition:removeNonMultipleUi} below, we may 
further
assume that $m_i \geq 2$, but this is not necessary for the inductive
argument to work.)

Let $j$ be such that $Q
\subseteq P_{\widehat{j}}$ and $\calQ_j$ (equivalently $\calU_j$) be of
least dimension; in other words, $j$ is the smallest element of $A$.
If $Q = P_{\widehat{j}}$ (i.e., $|A| = 1$), then the $\bigwedge^t \xi$ is
completely reducible, and we may use the Borel-Weil-Bott theorem to compute
the cohomology groups.  Hence suppose that $Q \neq P_{\widehat{j}}$; write 
$Q = Q' \cap P_{\widehat{j}}$
non-trivially, with $Q'$ being a parabolic subgroup.
Consider the diagram
\[
\xymatrix{
\GL_n/Q \ar[r]^-{\mathrm{p}_2}  \ar[d]^{\mathrm{p}_1} &
\GL_n/P_{\widehat{j}}\\
\GL_n/{Q'} }
\]
Note that $\bigwedge^t\xi$ decomposes as a direct sum of bundles of the 
form $(p_1)^*\eta \otimes (p_2)^*
(\bigwedge^{t_1}\calU_j^{\oplus m_j})$ where $\eta$ is a homogeneous bundle 
on $\GL_n/Q'$. We must compute $\homology^*(\GL_n/Q, (p_1)^*\eta \otimes 
(p_2)^* (\bigwedge^{t_1}\calU_j^{\oplus m_j}))$.
Using the Leray spectral sequence and the projection formula, we can
compute this from $\homology^*(\GL_n/Q', \eta \otimes R^*(p_1)_*(p_2)^*
(\bigwedge^{t_1}\calU_j^{\oplus m_j}))$. Now 
$\bigwedge^{t_1}\calU_j^{\oplus m_j}$, in turn, decomposes as a direct sum
of $\Schur_\mu \calU_j$, so we must compute $\homology^*(\GL_n/Q', \eta 
\otimes R^*(p_1)_*(p_2)^* \Schur_\mu \calU_j)$.
The Leray spectral sequence and the projection formula respect the various
direct-sum decompositions mentioned above. It would follow from
Proposition~\ref{proposition:higherdirectimageSpecific} below that for each
$\mu$, at most one of the $R^p(p_1)_*(p_2)^* \Schur_\mu \calU_j$ is
non-zero, so the abutment of the spectral sequence is, in fact, an
equality.

\begin{proposition}
\label{proposition:higherdirectimageSpecific}
With notation as above, let $\theta$ be a homogeneous bundle on
$\GL_n/P_{\widehat{j}}$. Then $R^i{p_1}_* {p_2}^*\theta$ is the locally 
free
sheaf associated to the vector-bundle $\GL_n \times^{Q'} \homology^i(Q'/Q, 
{p_2}^*\theta|_{Q'/Q})$ over $\GL_n/Q'$.
\end{proposition}

\begin{proof}
Follows from Proposition~\ref{proposition:higherdirectimage}.
\end{proof}

We hence want to determine the cohomology of the restriction of $\Schur_\mu
\calU_j$ on $Q'/Q$. It follows from the definition of $j$ that
$Q'/Q$ is a Grassmannian whose tautological quotient bundle and its dual 
are, respectively,  $\calQ_j|_{Q'/Q}$ and $\calU_j|_{Q'/Q}$.
We can therefore compute  $\homology^i(Q'/Q, \Schur_\mu
\calU_j|_{Q'/Q})$ using the Borel-Weil-Bott theorem.

\begin{example}
Suppose that $n=6$ and that $Q = P_{\widehat{\{2,4\}}}$. Then we have the
diagram
\[
\xymatrix{%
\GL_6/Q \ar[r]^-{\mathrm{p}_2}  \ar[d]^{\mathrm{p}_1} &
\GL_6/P_{\widehat{2}}\\
\GL_6/{P_{\widehat{4}}} }
\]
The fibre of $p_1$ is isomorphic to $P_{\widehat{4}}/Q$ which is a 
Grassmannian
of two-dimensional subspaces of a four-dimensional vector-space. Let $\mu =
(\mu_1, \mu_2)$ be a weight. Then we can compute the cohomology groups
$\homology^*(P_{\widehat{4}}/Q, \Schur_\mu \calU_2|_{P_{\widehat{4}}/Q})$
applying the Borel-Weil-Bott theorem~\cite[(4.1.5)]{WeymSyzygies03}
to the sequence $(0,0,\mu_1,\mu_2)$.  Note that
$\homology^*(P_{\widehat{4}}/Q, \Schur_\mu \calU_2|_{P_{\widehat{4}}/Q})$
is, if it is non-zero, $\Schur_\lambda W$ where $W$ is a four-dimensional 
vector-space that
is the fibre of the dual of the tautological quotient bundle of
$\GL_4/P_{\widehat{4}}$ and $\lambda$ is a partition with at most four
parts. Hence, by
Proposition~\ref{proposition:higherdirectimageSpecific}, we see that 
$R^i(p_1)_*(p_2)^* \Schur_\mu \calU_2$ is, if it is non-zero,
$\Schur_\lambda \calU_4$ on $\GL_6/P_{\widehat{4}}$.
\end{example}

We summarize the above discussion as a theorem:

\begin{theorem}
\label{theorem:ourtheorem}
For $w \in \calW_r$ the modules in the free resolution of
$\complex[Y_P(w)]$ given in Theorem~\ref{theorem:stepone}
can be computed.
\end{theorem}

We end this section with some observations.

\begin{proposition}
\label{proposition:removeNonMultipleUi}
Suppose that there exists $i$ such that $r+1 \leq i \leq n-1$ and such
that $\xi$ contains exactly one copy of $\calU_i$ as a direct summand. 
Let  $\xi' =
\calU_{i-1} \oplus \bigoplus_{\substack{j=1 \\ i_j \neq i}}^m \calU_{i_j}.
$
Then
$\homology^*(\GL_n/Q, \bigwedge^t\xi) = \homology^*(\GL_n/Q, 
\bigwedge^t\xi')$ for every $t$.
\end{proposition}

\begin{proof}
Note that $\xi'$ is a sub-bundle of $\xi$ with quotient
$\calU_{i}/\calU_{i-1}$. We claim that $\calU_{i}/\calU_{i-1} \simeq 
L_{\omega_{i-1} - \omega_{i}}$, where
for $1 \leq j \leq n$, $\omega_j$ is the $j$th fundamental weight.
Assume the claim. Then we have an exact sequence
\[
0 \to \bigwedge^t \xi'  \to \bigwedge^t \xi \to \bigwedge^{t-1} \xi' 
\otimes L_{\omega_{i-1} - \omega_{i}} \to 0
\]
Let $Q' = \bigcap\limits_{\substack{r \leq l \leq n-1 \\ l \neq i}}
P_{\hat l}$; then $Q = Q' \cap P_{\widehat{i}}$.
Let $p : \GL_n/Q \to \GL_n/Q'$ be the natural projection; its fibres are
isomorphic to $Q'/Q \simeq \GL_2/B_N \simeq \projective^1$. Note that 
$\bigwedge^{t-1} \xi' \otimes L_{\omega_{i-1}}$ is the pull-back along $p$
of some vector-bundle on $\GL_n/Q'$; hence it is constant on the fibres of
$p$. 

On the other hand, $L_{\omega_{i}}$ is the ample line bundle on 
$\GL_n/P_{\widehat{i}}$ that generates its Picard group, so 
$L_{-\omega_{i}}$ restricted to any fibre of $p$ is $\strSh(-1)$. Hence 
$\bigwedge^{t-1} \xi' \otimes L_{\omega_{i-1} - \omega_{i}}$ on any
fibre of $p$ is a direct sum of copies of $\strSh(-1)$ and hence it has no
cohomology. By Grauert's theorem~\cite[III.12.9]{HartAG},
$R^ip_*(\bigwedge^{t-1} \xi' \otimes L_{\omega_{i-1} - \omega_{i}}) =
0$ for every $i$, so, using the Leray spectral sequence, we conclude that 
$\homology^*(\GL_n/Q, \bigwedge^{t-1} \xi' \otimes L_{\omega_{i-1} -
\omega_{i}})=0$. This gives the proposition.

Now to prove the claim, note that $\calU_{i}/\calU_{i-1} \simeq 
\left(\calR_{n-i+1}/\calR_{n-i}\right)^*$.
Let $e_1, \ldots, e_n$ be a basis for $\complex^n$ such that the subspace
spanned by $e_i, \ldots, e_n$ is $B_N$-stable for every $1 \leq i \leq n$.
(Recall that we take the right action of $B_N$ on $\complex^n$.) Hence 
$\calR_{n-i+1}/\calR_{n-i}$ is the invertible sheaf on
which $B_N$ acts through the character $\omega_i-\omega_{i-1}$, which implies
the claim.
\end{proof}

\begin{remarkbox}[Determinantal case]
\label{remarkbox:determinantal}
Recall (see the paragraph after Definition~\ref{definition:ClassWk})
that $Y_P(w) = D_k$ if $w = (k+1, \ldots, n, N-k+1,
\ldots N) \in \calW_{n-k}$. In this case, $\calU_w = \calU_{n-k}^{\oplus 
(m-k+1)} \oplus
\bigoplus_{i=n-k+1}^{n-1} \calU_i$. Therefore
\[
\homology^*(\GL_n/Q_{n-k}, \bigwedge^*\xi) = \homology^*(\GL_n/Q_{n-k}, 
\bigwedge^* \calU_{n-k}^{\oplus m}) = \homology^*(\GL_n/P_{\widehat{n-k}}, 
\bigwedge^* \calU_{n-k}^{\oplus m})
\]
where the first equality comes from a repeated application of 
Proposition~\ref{proposition:removeNonMultipleUi} and the second one
follows by Lemma~\ref{lemma:pullback}, applied to the natural map $f : 
\GL_n/Q \to
\GL_n/P_{\widehat{n-k}}$.
Hence our approach recovers Lascoux's resolution of the determinantal
ideal~\cite{LascSyzVarDet78}; see also~\cite[Chapter~6]{WeymSyzygies03}.
\end{remarkbox}

\section{Examples}
\label{sec:examples}

We illustrate our approach with two examples. Firstly, we compute the
resolution of a determinantal variety using the inductive method from the
last section.

\begin{example}[$n \times m$ matrices of rank $\leq k$]
\label{example:4x3maximalminors}

If $k = 1$, then $w = (2, \ldots, n, n+m)$, and, hence, $\xi = 
\calU_{n-1}^{\oplus m}$.  Since this would not illustrate the inductive 
argument, let us take $k=2$.

Consider the ideal generated by the $3 \times 3$ minors of a $4 \times 3$
matrix of indeterminates. It is generated by four cubics, which have a
linear relation. Hence minimal free resolution of the quotient ring looks
like
\begin{equation}
\label{equation:4x3maximalminors}
0 \to R(-4)^{\oplus 3} \to R(-3)^{\oplus 4}  \to R \to 0.
\end{equation}

Note that $w  = (3,4,6,7)$ and $\xi = \calU_2^{\oplus 2} \bigoplus 
\calU_3$.
Write $G = \GL_4$ and $Q = P_{\widehat{2,3}}$.  Then $j=2$, $Q' = 
P_{\widehat{3}}$ and $Q'/Q \simeq \GL_3/P_{\widehat{2}}
\simeq \projective^2$.
Now there is a decomposition
\[
\bigwedge^t \xi = \bigoplus_{|\mu| \leq t} \Schur_{\mu'} \complex^2 \otimes
\Schur_{\mu} \calU_2 \otimes \bigwedge^{t-|\mu|} \calU_3
\]
Hence we need to consider only $\mu = (\mu_1, \mu_2) \leq (2,2)$.  On
$Q'/Q \simeq \GL_3/P_{\widehat{2}}$, we would apply
the Borel-Weil-Bott theorem~\cite[(4.1.5)]{WeymSyzygies03}
to the weight $(0,\mu_1, \mu_2)$ to compute the
cohomology of $\Schur_\mu \calU_j$. Thus we see that we need to consider 
only $\mu = (0,0)$, $\mu = (2,0)$ and $\mu = (2,1)$. From this, we
conclude that
\[
R^i(p_1)_*(p_2)^* (\Schur_{\mu'} \complex^2 \otimes \Schur_\mu \calU_2) = 
\begin{cases}
\strSh_{G/P_{\widehat{3}}}, & \text{if}\; i=0 \;\text{and}\; \mu=(0,0); \\
\bigwedge^2\calU_3, & \text{if}\; i=1 \;\text{and}\; \mu=(2,0);\\
(\bigwedge^3\calU_3)^{\oplus 2}, & \text{if}\; i=1 \;\text{and}\;
\mu=(2,1); \\
0, & \text{otherwise}.
\end{cases}
\]
We have to compute the cohomology groups of $(R^i(p_1)_*(p_2)^* 
(\Schur_{\mu'} \complex^2 \otimes \Schur_\mu \calU_2))
\otimes \bigwedge^{t-|\mu|}\calU_3$ on $G/P_{\widehat{3}}$. Now,
$\homology^*(G/P_{\widehat{3}},\bigwedge^i\calU_3) = 0$ for every $i>0$.
Further \begin{align*}
\bigwedge^2 \calU_3 \otimes \calU_3 & \simeq \bigwedge^3 \calU_3  \oplus 
\Schur_{2,1} \calU_3 && \text{for}\; \mu =(2,0) \;\text{and}\; t=3
\\
\bigwedge^2 \calU_3 \otimes \bigwedge^2 \calU_3 & \simeq \Schur_{2,1,1} 
\calU_3  \oplus \Schur_{2,2} \calU_3 && \text{for}\; \mu =(2,0) 
\;\text{and}\; t=4
\\
\bigwedge^2 \calU_3 \otimes \bigwedge^3 \calU_3 & \simeq \Schur_{2,2,1} 
\calU_3  && \text{for}\; (\mu =(2,0) \;\text{or}\; \mu=(2,1))
\;\text{and}\; t=5
\\
\bigwedge^3 \calU_3 \otimes \calU_3 & \simeq \Schur_{2,1,1} \calU_3 && 
\text{for}\; \mu =(2,1) \;\text{and}\; t=4
\\
\bigwedge^3 \calU_3 \otimes \bigwedge^3 \calU_3 & \simeq \Schur_{2,2,2} 
\calU_3
&& \text{for}\; \mu =(2,1) \;\text{and}\; t=6
\\
\end{align*}
Again, by applying the Borel-Weil-Bott 
theorem~\cite[(4.1.5)]{WeymSyzygies03} for $G/ P_{\widehat{3}}$, we
see that $\Schur_{2,2} \calU_3$, $\Schur_{2,2,1} \calU_3$ and 
$\Schur_{2,2,2} \calU_3$ have no cohomology. Therefore we conclude that
\[
\homology^j(G/Q, \bigwedge^t\xi) = \begin{cases}
\bigwedge^0\complex^{\oplus 4}, & \text{if}\; t=0 \;\text{and}\; j=0\\
\bigwedge^3\complex^{\oplus 4}, & \text{if}\; t=3 \;\text{and}\; j=2\\
(\bigwedge^4\complex^{\oplus 4})^{\oplus 3}, & \text{if}\; t=4 
\;\text{and}\; j=2\\
0, & \text{otherwise}.
\end{cases}
\]
These ranks agree with the expected ranks
from~\eqref{equation:4x3maximalminors}.
\end{example}

\begin{example}
\label{example:twoparabolics}
Let $n=6$, $m=6$, $k=4$ and $w=(5,6,8,9, 11,12)$. For this,
$Q = P_{\widehat{\{2,\cdots,5\}}}$ and $\calU_w = \calU_{2}^{\oplus 2} 
\oplus \calU_3 \oplus \calU_4^{\oplus 2}
\oplus \calU_5$.
After applying Propositions~\ref{proposition:removeUnnecParabolics} 
and~\ref{proposition:removeNonMultipleUi}, we reduce
to the situation  $Q = P_{\widehat{\{2,4\}}}$ and
$\xi = \calU_{2}^{\oplus 3} \oplus \calU_4^{\oplus 3}$.  Write $\xi = 
(\complex^3 \otimes_\complex \calU_2) \oplus (\complex^3\oplus \calU_4)$. 
Now we
project away from $\GL_6/P_{\widehat{2}}$.
\[
\xymatrix{%
\GL_6/Q \ar[r]^-{\mathrm{p}_2}  \ar[d]^{\mathrm{p}_1} &
\GL_6/P_{\widehat{2}}\\
\GL_6/{P_{\widehat{4}}} }
\]

The fibre of $p_1$ is isomorphic to $P_{\widehat{4}}/Q$ which is a 
Grassmannian
of two-dimensional subspaces of a four-dimensional vector-space.
We use the spectral sequence \begin{equation}
\label{equation:projectionss}
\homology^j(G/P_{\widehat{4}}, R^i{p_1}_*\bigwedge^t\xi) \Rightarrow 
\homology^{i+j}(G/Q, \bigwedge^t\xi).
\end{equation}

Observe that $\bigwedge^t \xi = \bigoplus_{t_1} \bigwedge^{t_1}  
(\complex^3 \otimes_\complex \calU_2)
\otimes \bigwedge^{t-t_1}  (\complex^3 \otimes_\complex \calU_4)$; the
above spectral sequence respects this decomposition. Further, using the
projection formula, we see that we need to compute
\[
\homology^j(G/P_{\widehat{4}},
(R^i{p_1}_* \bigwedge^{t_1} (\complex^3 \otimes_\complex \calU_2)) \otimes 
\bigwedge^{t-t_1}  (\complex^3 \otimes_\complex \calU_4)).
\]
Now, $R^i{p_1}_* \bigwedge^{t_1} (\complex^3 \otimes_\complex \calU_2)$ is 
the
vector-bundle associated to the $P_{\widehat{4}}$-module 
$\homology^i(P_{\widehat{4}}/Q, \bigwedge^{t_1} (\complex^3
\otimes_\complex \calU_2)|_{P_{\widehat{4}}/Q}) = 
\homology^i(P_{\widehat{4}}/Q, \bigwedge^{t_1} (\complex^3
\otimes_\complex \calU_2|_{P_{\widehat{4}}/Q}))$. Note that 
$\calU_2|_{P_{\widehat{4}}/Q}$ is the dual of the tautological quotient
bundle of $P_{\widehat{4}}/Q \simeq \GL_4/P_{\widehat{2}}$; we denote this
also, by abuse of notation, by $\calU_2$. Note, further, that 
$\bigwedge^{t_1} (\complex^3 \otimes_\complex \calU_2) = \bigoplus_{\mu 
\vdash t_1} \Schur_{\mu'} \complex^3 \otimes \Schur_{\mu}
\calU_2$. We need only consider $\mu \leq (3,3)$.
From the Borel-Weil-Bott theorem~\cite[(4.1.5)]{WeymSyzygies03}, it follows 
that
\[
\homology^i(P_{\widehat{4}}/Q, \Schur_{\mu} \calU_2) = \begin{cases}
\bigwedge^0(\complex^{\oplus^4}), & \text{if}\; i=0 \;\text{and}\;
\mu=(0,0);
\\
\bigwedge^3(\complex^{\oplus^4}), & \text{if}\; i=2 \;\text{and}\; 
\mu=(3,0);
\\
\bigwedge^4(\complex^{\oplus^4}), & \text{if}\; i=2 \;\text{and}\;
\mu=(3,1); \\
0, & \text{otherwise}.
\end{cases}
\]
Therefore we conclude that \[
R^i{p_1}_* \bigwedge^{t_1} (\complex^3 \otimes_\complex \calU_2) = 
\begin{cases}
\strSh_{\GL_4/P_{\widehat{2}}}, & \text{if}\; i=0 \;\text{and}\; t_1=0;
\\
\bigwedge^3\calU_4, & \text{if}\; i=2 \;\text{and}\; t_1=3; \\
(\bigwedge^4\calU_4)^{\oplus 3}, & \text{if}\; i=2 \;\text{and}\; t_1=4; \\
0, & \text{otherwise}.
\end{cases}
\]
Therefore for each pair $(t,t_1)$ at most one column of the summand of the
spectral sequence~\eqref{equation:projectionss} is non-zero; hence the 
abutment
in~\eqref{equation:projectionss} is in fact an equality.

Fix a pair $(t,t_1)$ and an integer $l$. Then we have
\begin{multline*}
\homology^{l}(G/Q, \bigwedge^t\xi) = \homology^{l}(G/P_{\widehat{4}}, 
\bigwedge^{t}  (\complex^3 \otimes \calU_4)) \oplus 
\homology^{l-2}(G/P_{\widehat{4}}, \bigwedge^3\calU_4 \otimes 
\bigwedge^{t-3}  (\complex^3 \otimes \calU_4))\\
 \oplus \homology^{l-2}(G/P_{\widehat{4}}, (\bigwedge^4\calU_4)^{\oplus 3} 
 \otimes \bigwedge^{t-4}  (\complex^3 \otimes \calU_4)).
\end{multline*}
Write $h^i(-) = \dim_\complex \homology^i(-)$.
Note that $\bigwedge^{t}  (\complex^3 \otimes \calU_4)
\simeq \bigoplus_{\lambda \vdash t} \Schur_{\lambda'} \complex^3 \otimes 
\Schur_{\lambda} \calU_4$, by the Cauchy formula.  Write $d_{\mu'} = 
\dim_\complex
\Schur_{\mu'}\complex^{\oplus3}$. Thus, from the above
equation, we see, that for every $l$ and for every
$t$,
\begin{equation}
\label{equation:cohDecMainExample}
h^l(\wedge^t \xi) = \sum_{\mu \vdash t} d_{\mu'} h^l(\Schur_{\mu} \calU_4)
+ \sum_{\mu \vdash t-3} d_{\mu'} h^{l-2}(\wedge^3 \calU_4 \otimes 
\Schur_{\mu} \calU_4)
+ 3\sum_{\mu \vdash t-4} d_{\mu'} h^{l-2}(\wedge^4 \calU_4 
\otimes\Schur_{\mu} \calU_4)
\end{equation}
(Here the cohomology is calculated over $\GL_6/Q$ on the left-hand-side and
over $\GL_6/P_{\widehat{4}}$ on the right-hand-side.) 
For any $\mu$, if $d_{\mu'} \neq 0$, then $\mu_1 \leq 3$.
Any $\mu$ that contributes a non-zero integer to the right-hand-side 
of~\eqref{equation:cohDecMainExample} has at
most four parts and $m_1 \leq 3$.
Further, if $\Schur_{\lambda} \calU_4$ is an irreducible summand of a
representation on the right-hand-side
of~\eqref{equation:cohDecMainExample} with non-zero cohomology, then
$\lambda$ has at most four
parts and is such that $\lambda_1 \leq 4$. Therefore for $\lambda \leq
(4,4,4,4)$, we compute the cohomology using
the Borel-Weil-Borel theorem:
\[
\homology^{i}(G/P_{\widehat{4}}, \Schur_{\lambda} \calU_4) = \begin{cases}
\wedge^0(\complex^{\oplus 6}), & \text{if}\; i=0 \;\text{and}\; \lambda =0;
\\
\Schur_{(\lambda_1-2,1,1,\lambda_2,\lambda_3, \lambda_4)}(\complex^{\oplus 
6}),
& \text{if}\; i=2,
\lambda_1 \in \{3,4\} \;\text{and}\; (\lambda_2,\lambda_3, \lambda_4) \leq 
(1,1,1); \\

\Schur_{(2,2,2,2,\lambda_3, \lambda_4)}(\complex^{\oplus 6}),
& \text{if}\; i=4,
\lambda_1 = \lambda_2 =4 \;\text{and}\; (\lambda_3, \lambda_4) \leq (2,2); 
\\
0, & \text{otherwise}.
\end{cases}
\]
We put these together to compute $h^l(\wedge^t\xi)$;
the result is listed in Table~\ref{table:mainexample}. From this we get the
following resolution:
\[
\xymatrix@C=5mm{
0 \ar[r] & R(-12)^{26} \ar[r]  & R(-11)^{108} \ar[r]  &
{\def\arraystretch{0.1}
\begin{matrix} R(-6)^{10} \\ \oplus \\ R(-10)^{153} \end{matrix}}
\ar[r] &
{\def\arraystretch{0.1}
\begin{matrix} R(-5)^{36} \\ \oplus \\ R(-7)^{36} \\ \oplus \\ R(-9)^{70} 
\end{matrix}}
\ar[r] &
{\def\arraystretch{0.1}
\begin{matrix} R(-3)^{45} \\ \oplus \\ R(-5)^{53} \end{matrix}}
\ar[r] &
{\def\arraystretch{0.1}
\begin{matrix} R(-2)^{20} \\ \oplus \\ R(-4)^{18} \end{matrix}}
\ar[r] & R \ar[r]& 0.
}
\]
Note, indeed, that $\dim Y_{Q}(w) = \dim X_{Q}(w) = 4+4+5+5+6+6 = 30$
and that $\dim O^\mhyphen_{\GL_N/P} = 6\cdot 6 = 36$, so the codimension is $6$. 
Since
the variety is Cohen-Macaulay, the length of a minimal free resolution is
$6$.

\begin{table}
\begin{tabular}{|c|c|c|c|c|c|c|c|}
\hline
$t$ & $h^0(\wedge^t\xi)$ & $h^1(\wedge^t\xi)$ & $h^2(\wedge^t\xi)$ & 
$h^3(\wedge^t\xi)$ & $h^4(\wedge^t\xi)$ & $h^5(\wedge^t\xi)$ & 
$h^6(\wedge^t\xi)$ \\
\hline
0 & 1 & 0 & 0 & 0 & 0 & 0 & 0 \\
\hline
1 & 0 & 0 & 0 & 0 & 0 & 0 & 0 \\
\hline
2 & 0 & 0 & 0 & 0 & 0 & 0 & 0 \\
\hline
3 & 0 & 0 & 20 & 0 & 0 & 0 & 0 \\
\hline
4 & 0 & 0 & 45 & 0 & 0 & 0 & 0 \\
\hline
5 & 0 & 0 & 36 & 0 & 18 & 0 & 0 \\
\hline
6 & 0 & 0 & 10 & 0 & 53 & 0 & 0 \\
\hline
7 & 0 & 0 & 0 & 0 & 36 &  0 & 0 \\
\hline
8 & 0 & 0 & 0 & 0 & 0 & 0 & 0 \\
\hline
9 & 0 & 0 & 0 & 0 & 0 & 0& 70  \\
\hline
10 & 0 & 0 & 0 & 0 & 0 & 0 & 153 \\
\hline
11 & 0 & 0 & 0 & 0 & 0 & 0 & 90 \\
\hline
12 & 0 & 0 & 0 & 0 & 0 & 0 & 26 \\
\hline
\end{tabular}
\caption{Ranks of the relevant cohomology groups}
\label{table:mainexample}
\end{table}
\end{example}

\section{Further remarks}
\label{sec:furtherrmks}

\subsection*{A realization of Lascoux's resolution for determinantal
varieties}

We already saw in Remark~\ref{remarkbox:determinantal} that when
$Y_P(w)=D_k$, computing
$\homology^*(\GL_n/Q_{n-k}, \bigwedge^*\xi)$ is reduced,
by a repeated application of
Proposition~\ref{proposition:removeNonMultipleUi}
to computing the cohomology groups of (completely reducible) vector bundles
on the Grassmannian $\GL_n/P_{\widehat{n-k}}$. We thus realize Lascoux's resolution of the determinantal variety using our approach.

In this section, we give yet another desingularization of $D_k$ (for a
suitable choice of the parabolic subgroup) so that the variety $V$ 
of Diagram~\eqref{equation:KLWdiagram} 
is in fact a Grassmannian. 
Recall (the paragraph after Definition~\ref{definition:ClassWk} or
Remark~\ref{remarkbox:determinantal})
that $Y_P(w) = D_k$ if $w = (k+1, \ldots, n, N-k+1,
\ldots N) \in \calW_{n-k}$.
Let $\tilde{P} = P_{\widehat{\{n-k,n\}}} \subseteq \GL_N$.
Let $\tilde w$ be the representative of
the coset $w\tilde P$ in $W^{\tilde P}$.

\begin{proposition}
$X_{{\tilde{P}}}(\tilde w)$ is smooth and the natural map
$X_{{\tilde{P}}}(\tilde w)
\to X_P(w)$ is proper and birational, i.e, $X_{{\tilde{P}}}(\tilde w)$ is a 
desingularization of $X_P(w)$.
\end{proposition}

\begin{proof}
The proof is similar to that of
Proposition~\ref{proposition:smoothXtildew}.
Let $w_{\mathrm{max}}  = 
(k+1, \ldots, n, N-k+1, \ldots N, N-k, \ldots, n+1, k, \ldots, 1) \in W$.
Then
$X_{B_N}(w_{\mathrm{max}})$ is the inverse image of 
$X_{{\tilde{P}}}(\tilde w)$ under the
natural morphism 
$\GL_N/B_N \to \GL_N/{{\tilde{P}}}$, and that $w_{\mathrm{max}}$ is a
$4231$ and $3412$-avoiding element of $W = S_N$.
\end{proof}

We have $P/{\tilde{P}}\cong \GL_n/P_{\widehat{n-k}}$.  As in 
Section~\ref{sec:schubertdesing}, we have the following. 
Denoting by $Z$ the preimage inside $X_{{\tilde{P}}}(\tilde w)$ of $Y_P(w)$ (under the restriction to $X_{{\tilde{P}}}(\tilde w)$ of the natural projection 
$G/{\tilde{P}}\rightarrow G/P$), we have $Z\subset O^\mhyphen\times P/{\tilde{P}}$,
and the image of $Z$ under the
second projection is $V:=P/{\tilde{P}}(\cong \GL_n/P_{\widehat{n-k}})$. The 
inclusion $Z \hookrightarrow O^\mhyphen  \times V$ is a sub-bundle (over $V$) of 
the trivial bundle $O^\mhyphen \times V$.
Denoting by $\xi$ the dual of the quotient bundle on $V$ corresponding to
$Z$, we have that the homogeneous bundles $\bigwedge^{i+j} \xi$ on 
$\GL_n/P_{\widehat{n-k}}$ are completely reducible, and hence may be 
computed using Bott's algorithm.

\subsection*{Multiplicity}
We describe how the free resolution obtained in
Theorem~\ref{theorem:stepone} can be used to get an expression for
the multiplicity 
${\mathrm{mult}_{\mathrm{id}}(w)}$ of the local ring of the Schubert variety 
$X_P(w) \subseteq \GL_N/P$ at the point $e_{\mathrm{id}}$. Notice that 
$Y_P(w)$ is an affine neighbourhood of $e_{\mathrm{id}}$. 
We noticed in Section~\ref{sec:freeresolutions} that
$Y_{P}(w)$ is a closed subvariety of $O^\mhyphen_{\GL_N/P}$ defined by
homogeneous equations. In $O^\mhyphen_{\GL_N/P}$, $e_{\mathrm{id}}$ is the origin;
hence in $Y_{P}(w)$ it is defined by the unique homogeneous
maximal ideal of $\complex[Y_P(w)]$. Therefore 
$\complex[Y_P(w)]$ is the associated graded ring of the local ring of 
$\complex[Y_P(w)]$ at $e_{\mathrm{id}}$ (which is also the local ring
of $X_P(w)$ at $e_{\mathrm{id}}$). Hence 
${\mathrm{mult}_{\mathrm{id}}(w)}$ is the normalized leading coefficient of
the Hilbert series of $\complex[Y_P(w)]$.

Observe that the Hilbert series of $\complex[Y_P(w)]$ can be obtained as an
alternating sum of the Hilbert series of the modules $F_i$ in 
Theorem~\ref{theorem:stepone}. Write 
$h^j(-) = \dim_\complex \homology^j(X_{Q_s}(w'), -)$ for coherent sheaves
on $X_{Q_s}(w')$. Then Hilbert series of $\complex[Y_P(w)]$ is
\begin{equation}
\label{equation:hilbSeries}
\frac%
{\sum\limits_{i=0}^{mn} \sum\limits_{j=0}^{\dim X_{Q_s}(w')}(-1)^i
h^j\left(\bigwedge^{i+j}\calU_w\right)t^{i+j}}
{(1-t)^{mn}}.
\end{equation}
We may harmlessly change the range of summation in the numerator
of~\eqref{equation:hilbSeries} to $-\infty < i,j < \infty$; this is
immediate for $j$, while for $i$, we note that the proof of
Theorem~\ref{theorem:geometrictechnique} implies that
$h^j\left(\bigwedge^{i+j}\calU_w\right) = 0$ for every $i<0$ and for every
$j$. Hence we may write the numerator 
of~\eqref{equation:hilbSeries} as (with $k=i+j$)
\begin{equation}
\label{equation:hilbSeriesNum}
\sum\limits_{k=0}^{\infty}(-1)^kt^k \sum\limits_{j=0}^{\infty}
(-1)^j h^j\left(\wedge^{k}\calU_w\right)
=
\sum\limits_{k=0}^{\rank \calU_w}(-1)^k \chi\left(\wedge^{k}\calU_w\right)t^k.
\end{equation}

Since $\wedge^{k}\calU_w$ is also a $T_n$-module, where $T_n$ is the
subgroup of diagonal matrices in $\GL_n$, one may decompose 
$\wedge^{k}\calU_w$ as a sum of rank-one $T_n$-modules and 
use the Demazure character formula to compute the Euler characteristics
above.

It follows from generalities on Hilbert series (see, e.g.,~\cite
[Section~4.1]{BrHe:CM}) that the polynomial
in~\eqref{equation:hilbSeriesNum} is divisible by $(1-t)^c$ where $c$ is
the codimension of $Y_P(w)$ in $O^\mhyphen_{\GL_N/P}$, and that after we
divide it and substitute $t=1$ in the quotient, we get
${\mathrm{mult}_{\mathrm{id}}(w)}$.
This gives an expression for 
${e_{\mathrm{id}}(w)}$ apart from those of 
\cite{LakshmibaiWeymanMultiplicity1990, KreimanLakshmibaiMultiplicity2004}.

\subsection*{Castelnuovo-Mumford Regularity}

Since $\complex[Y_{P}(w)]$ is a graded quotient ring of
$\complex[O^\mhyphen_{\GL_N/P}]$, it defines a coherent sheaf over the 
corresponding projective space $\projective^{mn-1}$. 

Let $F$ be a coherent sheaf on $\projective^n$.
The \define{Castelnuovo-Mumford regularity} of $F$ (with respect to 
$\strSh_{\projective^{n}}(1)$) is the smallest integer $r$ such that 
$\homology^i(\projective^{n}, F \otimes \strSh_{\projective^{n}}(r-i)) = 0$ 
for every $1 \leq i \leq n$; we denote it
by $\reg F$. Similarly, if $R = \Bbbk[x_0, \ldots, x_n]$ be a polynomial 
ring over a field $\Bbbk$
with $\deg x_i = 1$ for every $i$ and $M$ is a finitely generated graded
$R$-module, the \define{Castelnuovo-Mumford regularity} of $M$ to
be the smallest integer $r$ such that $\left(\homology_{(x_0, \ldots, 
x_n)}^i(M)\right)_{r+1-i} = 0$  for every $0
\leq i \leq n+1$; we denote it by $\reg M$.
(Here $\homology_{(x_0, \ldots, x_n)}^i(M)$ is the $i$ \emph{local
cohomology} module of $M$, and is a graded $R$-module.) It is known that
$\reg F = \reg \left(\oplus_{i \in \ints} \homology^0(\projective^{n}, F 
\otimes \strSh_{\projective^{n}}(i))\right)$
for every coherent sheaf $F$ and that if $\depth M \geq 2$, then $\reg M =
\reg \widetilde M$.
See~\cite[Chapter~4]{EisSyz05} for details.

\begin{proposition}
In the notation of Diagram~\eqref{equation:genericKLW}, $\reg \complex[Y] = 
\max \{ j : \homology^j(V, \wedge^{*}\xi) \neq 0\}$.
\end{proposition}

\begin{proof}
Let $R = \complex[\bbA]$. It is known that
$\reg M = \max \{j : \Tor_{i}^R(\Bbbk,M)_{i+j} \neq 0 \;\text{for some}\;
i\}$;
see~\cite[Chapter~4]{EisSyz05} for a proof.  The proposition now follows 
from noting that $\Tor_i^R(\complex, \complex[Y])_{i+j} \simeq
\homology^j(V, \wedge^{i+j}\xi)$ (Theorem~\ref{theorem:stepone}).
\end{proof}

Now let $w = (n-r+1, n-r+2, \ldots, n, a_{r+1}, \ldots, a_{n-1}, N) \in
\calW_r$. We would like to determine $\reg \complex[Y_P(w)] = \max \{ j : 
\homology^j(\GL_n/Q_r, \wedge^{*}\calU_w) \neq 0\}$.
Let $a_r = n$ and $a_n = N$. For $r \leq i \leq n-1$,
define $m_i = a_{i+1}-a_i$. Note that $\calU_i$ appears in $\calU_w$
with multiplicity $m_i$ and that $m_i > 0$.
Based on the examples that we have calculated, we have the following
conjecture.  \begin{conjecture}
\label{conjecture:regularity}
With notation as above,
\[
\reg \complex[Y_{P}(w)] = \sum_{i=r}^{n-1} \left(m_i-1\right)i.
\]
\end{conjecture}
(Note that since $Y_P(w)$ is Cohen-Macaulay, $\reg \complex[Y_{P}(w)] = 
\reg \strSh_{Y_{P}(w)}$.)
Consider the examples in Section~\ref{sec:examples}.  In 
Example~\ref{example:4x3maximalminors}, $m_2=2$, $m_3=1$ and $\reg 
\complex[Y_{P}(w)] = (2-1)2+0 = 2$. In
Example~\ref{example:twoparabolics}, $m_2=m_4=2$ and $m_3=m_5=1$, so $\reg 
\complex[Y_{P}(w)] = (2-1)2 + 0 + (2-1)4 + 0 = 6$, which in
deed is the case, as we see from Table~\ref{table:mainexample}.

\ifreadkumminibib
\bibliographystyle{alphabbr}
\bibliography{kummini}
\else

\def\cfudot#1{\ifmmode\setbox7\hbox{$\accent"5E#1$}\else
  \setbox7\hbox{\accent"5E#1}\penalty 10000\relax\fi\raise 1\ht7
  \hbox{\raise.1ex\hbox to 1\wd7{\hss.\hss}}\penalty 10000 \hskip-1\wd7\penalty
  10000\box7}

\fi 

\end{document}